\documentclass[11pt,reqno]{amsart}

\usepackage{amssymb,amsmath,amsthm,epsfig,bbold,amscd,thmtools,setspace,
mathtools,indentfirst,graphics,color,float,graphicx}

\usepackage{enumerate}
\usepackage{cases}
\usepackage{appendix}

\newtheorem{thm}{Theorem}[section]
\newtheorem{lem}{Lemma}[section]

\theoremstyle{definition}

\newtheorem{rem}{Remark}[section]
\newtheorem{ques}{Question}[section]

\def\R{{\mathbb R}}

\numberwithin{equation}{section}

\begin{document}

\title[Parabolic Choquard-Pekar Inequalities]{Initial Pointwise Bounds
  and Blow-up for Parabolic Choquard-Pekar Inequalities}

%    Remove any unused author tags.

%    author one information
\author{Steven D. Taliaferro}
\address{Mathematics Department, Texas A\&M
    University, College Station, TX 77843-3368}
%\curraddr{}
\email{stalia@math.tamu.edu}
%\thanks{}

%    author two information
%\author{}
%\address{}
%\curraddr{}
%\email{}
%\thanks{}



\date{}

\dedicatory{}

\begin{abstract}
We study the behavior as $t\to 0^+$ of nonnegative functions
\begin{equation}\label{0.1}
 u\in C^{2,1} (\mathbb{R}^n\times (0,1))
\cap L^\lambda (\mathbb{R}^n\times (0,1)),\quad n\ge 1,
\end{equation}
satisfying the parabolic Choquard-Pekar type inequalities 
\begin{equation}\label{0.2}
 0\leq u_t-\Delta u\leq(\Phi^{\alpha/n}*u^\lambda )u^\sigma \quad
\text{ in }B_1 (0)\times (0,1)
\end{equation}
where $\alpha\in(0,n+2)$, $\lambda>0$, and $\sigma\geq0$ are constants,
$\Phi$ is the heat kernel, and
$*$ is the convolution operation in $\mathbb{R}^n\times (0,1)$.  We provide
optimal conditions on $\alpha,\lambda$, and $\sigma$ such that
nonnegative solutions $u$ of \eqref{0.1},\eqref{0.2} satisfy
pointwise bounds in compact subsets of $B_1(0)$ as $t\to 0^+$. We
obtain similar results for nonnegative solutions of
\eqref{0.1},\eqref{0.2} when $\Phi^{\alpha/n}$ in \eqref{0.2} is
replaced with the fundamental solution $\Phi_\alpha$ of the fractional
heat operator $(\frac{\partial}{\partial t}-\Delta)^{\alpha/2}$.
\end{abstract}

\maketitle

\section{Introduction}\label{sec1}
In this paper we study the behavior as $t\to 0^+$ of nonnegative
functions  
\begin{equation}\label{1.1}
 u\in C^{2,1} (\mathbb{R}^n\times (0,T))
\cap L^\lambda (\mathbb{R}^n\times (0,T)),\quad n\ge 1,
\end{equation}
satisfying the nonlocal parabolic Choquard-Pekar type inequalities 
\begin{equation}\label{1.2}
 0\leq Hu\leq(\Phi^{\alpha/n}*u^\lambda )u^\sigma \quad
\text{ in }\Omega\times (0,T)
\end{equation}
where $\alpha\in(0,n+2)$, $\lambda>0$, $\sigma\geq0$, and $T>0$ are constants,
$\Omega$ is an open subset of $\mathbb{R}^n$, $Hu=u_t-\Delta u$ is
the heat operator,  
\begin{equation}\label{1.3}
  \Phi(x,t)=\begin{cases}
\frac{1}{(4\pi t)^{n/2}} 
e^{-\frac{|x|^2}{4t}}&\text{for }(x,t)\in \mathbb{R}^n\times (0,\infty)\\
   0&\text{for }(x,t)\in \mathbb{R}^n\times (-\infty,0]
   \end{cases}             
  \end{equation}
  is the heat kernel, and $*$ is the convolution operation in
$\mathbb{R}^n\times(0,T)$, that is,
\[
(\Phi^{\alpha/n}*u^\lambda)(x,t)
=\iint_{\R^n\times(0,T)}\Phi(x-y,t-s)^{\alpha/n}u(y,s)^\lambda dy\,ds.
\]
The regularity condition $u\in
  L^\lambda(\mathbb{R}^n\times (0,T))$ in \eqref{1.1} and the upper
  bound of $n+2$ for $\alpha$ are natural because one does not want
  the nonlocal convolution operation on the right side of \eqref{1.2}
  to be infinite at every point in $\mathbb{R}^n\times (0,T)$.

We also obtain results on the behavior as $t\to 0^+$ of nonnegative
solutions of \eqref{1.1},\eqref{1.2} when $\Phi^{\alpha/n}$ in \eqref{1.2} is
replaced with the fundamental solution $\Phi_\alpha$ of the fractional
heat operator $(\frac{\partial}{\partial t}-\Delta)^{\alpha/2}$. (See
Remark \ref{rem2}.)

A motivation for the study of \eqref{1.1},\eqref{1.2} comes from the
nonlocal elliptic equation
\begin{equation}\label{prototype}
-\Delta u=(\Gamma^{\alpha/(n-2)}* u^{\lambda}) |u|^{\lambda-2}u\quad
\text{ in }\mathbb{R}^n,
\end{equation}
where $\alpha\in (0,n)$, $\lambda>1$ and $\Gamma(x)=C(n)/|x|^{n-2}$ is
a fundamental solution of $-\Delta$. For $n=3$, $\alpha=1$,
and $\lambda=2$, equation \eqref{prototype} is known in the literature as the
{\it Choquard-Pekar equation} and was introduced in \cite{P1954} as a
model in quantum theory of a polaron at rest (see also
\cite{DA2010}). Later, the equation \eqref{prototype} appears as a model
of an electron trapped in its own hole, in an approximation to
Hartree-Fock theory of one-component plasma \cite{L1976}.  More
recently, the same equation \eqref{prototype} was used in a model of
self-gravitating matter (see, e.g., \cite{J1995,MPT1998}) and it is
known in this context as the {\it Schr\"odinger-Newton equation}.

The Choquard-Pekar equation \eqref{prototype} has been investigated
for a few decades by variational methods starting with the pioneering
works of Lieb \cite{L1976} and Lions \cite{Lions1980,Lions1984}. More
recently, new and improved techniques have been devised to deal with
various forms of \eqref{prototype} (see, e.g.,
\cite{MZ2010,MZ2012,MV2013a,MV2013b,MV2015,WW2009} and the references
therein).

Using nonvariational methods, the authors in \cite{MV2013b} obtained
sharp conditions for the nonexistence of nonnegative solutions to
$$-\Delta u \geq (\Gamma^{\alpha/(n-2)}* u^{\lambda}) u^{\sigma}$$
in an exterior domain of $\mathbb{R}^n$, $n\geq 3$.

For some very recent results on positive solutions Choquard-Pekar
equations and inequalities which have an isolated singularity at the
origin see \cite{CZ} and \cite{GT2016-2}.

Other examples of nonlocal equations which have been studied
extensively in recent years are equations containing the fractional
Laplacian and some of these equations are equivalent to equations
containing convolutions with powers of the fundamental solution
$\Gamma$ of $-\Delta u$. For example, see \cite{Zhuo2016} and
\cite{MCL2011}.

On the other hand, we know of no results for nonlocal equations or
inequalities when the nonlocal feature of the problem is due to
convolutions with powers of the fundamental solution \eqref{1.3} of
the heat equation. Our results for \eqref{1.1},\eqref{1.2} are, in
this regard, new.

In this paper we consider the following question.

\begin{ques}\label{ques1}
  Suppose $\alpha\in(0,n+2)$ and $\lambda>0$ are constants and
  $\Omega$ is an open subset of $\R^n$, $n\ge 1$.  For which
  nonnegative constants $\sigma$, if any, does there exist a
  continuous function $\varphi:(0,1)\to(0,\infty)$ such that for all
  compact subsets $K$ of $\Omega$ and all nonnegative solutions $u$ of
  \eqref{1.1},\eqref{1.2} we have
\begin{equation}\label{1.4}
 \max_{x\in K}u(x,t)=O(\varphi(t))\quad\text{ as }t\to0^+
\end{equation}
and what is the optimal such $\varphi$ when it exists?
\end{ques}

We call the function $\varphi$ in \eqref{1.4} a pointwise bound for
$u$ on compact subsets of $\Omega$ as $t\to0^+$.

\begin{rem}\label{rem1}
  Suppose $0<\lambda<(n+2)/n$. Then, since $u=\Phi$, where $\Phi$ is
  the heat kernel given by \eqref{1.3}, is a solution of
  \eqref{1.1},\eqref{1.2} and $\Phi(0,t)=(4\pi t)^{-n/2}$, we see that
  any pointwise bound for nonnegative solutions $u$ of
  \eqref{1.1},\eqref{1.2} on compact subsets of $\Omega$ as
  $t\to0^+$ must be at least as large as $t^{-n/2}$ and whenever
  $t^{-n/2}$ is such a bound it is necessarily optimal.
\end{rem}

In order to state our results for Question \ref{ques1}, we define for each
$\alpha\in(0,n+2)$ the continuous, piecewise linear functions $g_\alpha,
G_\alpha:(0,\infty)\to [0,\infty)$ by
\begin{equation}\label{1.5}
 g_\alpha (\lambda)=
 \begin{cases}
  \frac{n+2}{n} & \text{if }0<\lambda<\frac{n+2-\alpha}{n}\\
  \frac{2(n+2)-\alpha}{n}-\lambda & \text{if }\frac{n+2-\alpha}{n}
\leq\lambda<\frac{n+2}{n}\\
  \max\{0,1-\frac{\alpha-2}{n+2}\lambda \} & \text{if }
\lambda\geq\frac{n+2}{n}
 \end{cases}
\end{equation}
and
\[
 G_\alpha (\lambda)=
 \begin{cases}
  \frac{2(n+2)-\alpha}{n}-\lambda & \text{if }
0<\lambda<\frac{n+2}{n}\\
  \max\{0,1-\frac{\alpha-2}{n+2}\lambda \} & \text{if }
\lambda\geq\frac{n+2}{n}.
 \end{cases}
\]
These functions are graphed in Figure \ref{fig1} (resp. Figure
\ref{fig2}) when $\alpha\in (2,n+2)$ (resp. $\alpha\in (0,2]$). 
Note that  
\[
g_\alpha(\lambda)=G_\alpha(\lambda)
\quad\text{for }\frac{n+2-\alpha}{n}\le\lambda<\infty
\]
and
\[
g_\alpha(\lambda)<G_\alpha(\lambda)
\quad\text{for }0<\lambda<\frac{n+2-\alpha}{n}.
\]

According to the following theorem, if the point $(\lambda,\sigma)$
lies below the graph of $\sigma=g_\alpha(\lambda)$ then there exists a
pointwise bound for nonnegative solutions $u$ of
\eqref{1.1},\eqref{1.2} on compact subsets of $\Omega$ as $t\to 0^+$ .

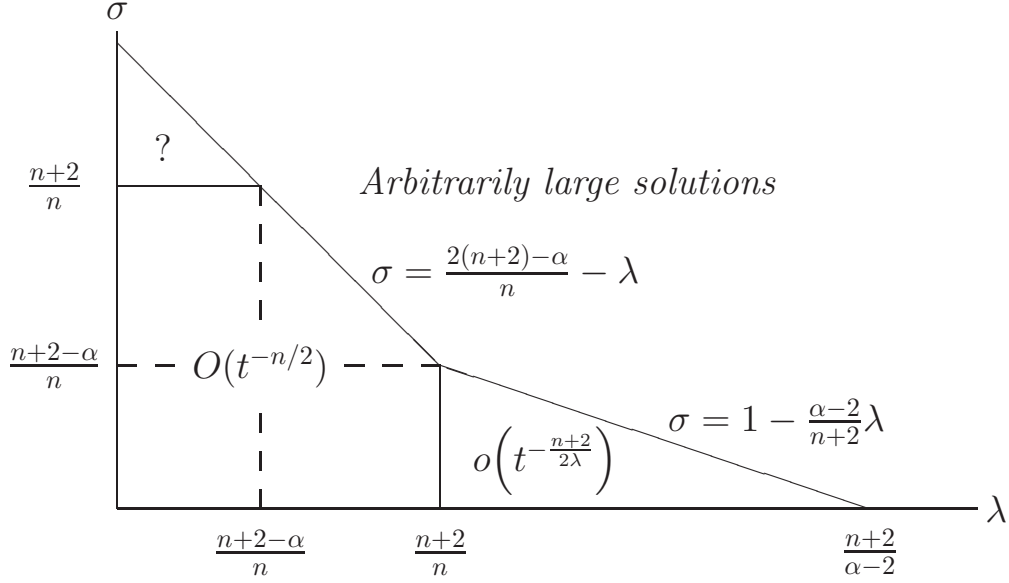
\begin{figure} \Large 

\setlength{\unitlength}{1.0in}

\begin{picture}(6,3.2)(-0.2,0)
\put(0.75,0.5){\line(0,1){2.50}}
\put(0.75,0.5){\line(1,0){4.50}}    
\put(0.75,2.19){\line(1,0){0.75}}
\put(0.75,2.94){\line(1,-1){1.69}}
\put(2.44,0.50){\line(0,1){0.75}}
\put(4.69,0.50){\line(-3,1){2.25}}

\put(1.20,0.12){\makebox(0.6,0.3){$\frac{n+2-\alpha}{n}$}}
\put(2.14,0.12){\makebox(0.6,0.3){$\frac{n+2}{n}$}}
\put(4.39,0.12){\makebox(0.6,0.3){$\frac{n+2}{\alpha-2}$}}
\put(0.12,1.10){\makebox(0.6,0.3){$\frac{n+2-\alpha}{n}$}}
\put(0.12,2.04){\makebox(0.6,0.3){$\frac{n+2}{n}$}}

\multiput(1.50,0.50)(0,0.2){3}{\line(0,1){0.1}}
\multiput(1.50,2.19)(0,-0.2){4}{\line(0,-1){0.1}}

\multiput(0.75,1.25)(0.2,0){2}{\line(1,0){0.1}}
\multiput(2.44,1.25)(-0.2,0){3}{\line(-1,0){0.1}}

\put(2.54,0.6){\makebox(0.9,0.3){$o\Bigl(t^{-\frac{n+2}{2\lambda}}\Bigr)$}}
\put(1.1,1.10){\makebox(0.8,0.3){$O(t^{-n/2})$}}
\put(3.6,0.8){\makebox(1.2,0.3){$\sigma=1-\frac{\alpha-2}{n+2}\lambda$}}
\put(0.89,2.29){\makebox(0.2,0.2){?}}
\put(2.03,1.6){\makebox(1.5,0.3){$\sigma=\frac{2(n+2)-\alpha}{n}-\lambda$}}
\put(1.95,2.05){\makebox(2.3,0.3){\it Arbitrarily large solutions}}

\put(5.25,0.40){\makebox(0.2,0.2){$\lambda$}}
\put(0.65,3.01){\makebox(0.2,0.2){$\sigma$}}
\end{picture}

\vskip -0.25in
\caption{Case $\alpha\in (2,n+2)$.}
\label{fig1}
\end{figure}

\begin{figure} \Large 

\setlength{\unitlength}{1.0in}

\begin{picture}(6,3.2)(-0.2,0)
\put(0.75,0.5){\line(0,1){2.50}}
\put(0.75,0.5){\line(1,0){4.50}}    
\put(0.75,2.19){\line(1,0){0.75}}
\put(0.75,2.94){\line(1,-1){1.69}}
\put(2.44,0.50){\line(0,1){0.75}}
\put(2.44,1.25){\line(3,1){2.75}}

\put(1.20,0.12){\makebox(0.6,0.3){$\frac{n+2-\alpha}{n}$}}
\put(2.14,0.12){\makebox(0.6,0.3){$\frac{n+2}{n}$}}
\put(0.12,1.10){\makebox(0.6,0.3){$\frac{n+2-\alpha}{n}$}}
\put(0.12,2.04){\makebox(0.6,0.3){$\frac{n+2}{n}$}}

\multiput(1.50,0.50)(0,0.2){3}{\line(0,1){0.1}}
\multiput(1.50,2.19)(0,-0.2){4}{\line(0,-1){0.1}}

\multiput(0.75,1.25)(0.2,0){2}{\line(1,0){0.1}}
\multiput(2.44,1.25)(-0.2,0){3}{\line(-1,0){0.1}}

\put(2.90,0.75){\makebox(0.9,0.3){$o\Bigl(t^{-\frac{n+2}{2\lambda}}\Bigr)$}}
\put(1.1,1.10){\makebox(0.8,0.3){$O(t^{-n/2})$}}
\put(3.6,1.4){\makebox(1.2,0.3){$\sigma=1+\frac{2-\alpha}{n+2}\lambda$}}
\put(0.89,2.29){\makebox(0.2,0.2){?}}
\put(1.73,1.9){\makebox(1.5,0.3){$\sigma=\frac{2(n+2)-\alpha}{n}-\lambda$}}
\put(1.95,2.35){\makebox(2.3,0.3){\it Arbitrarily large solutions}}

\put(5.25,0.40){\makebox(0.2,0.2){$\lambda$}}
\put(0.65,3.01){\makebox(0.2,0.2){$\sigma$}}
\end{picture}

\vskip -0.25in
\caption{Case $\alpha\in (0,2]$. When $\alpha=2$ the graph on the
  interval $\lambda>(n+2)/n$ is the horizontal half line $\sigma=1$.}
\label{fig2}
\end{figure}
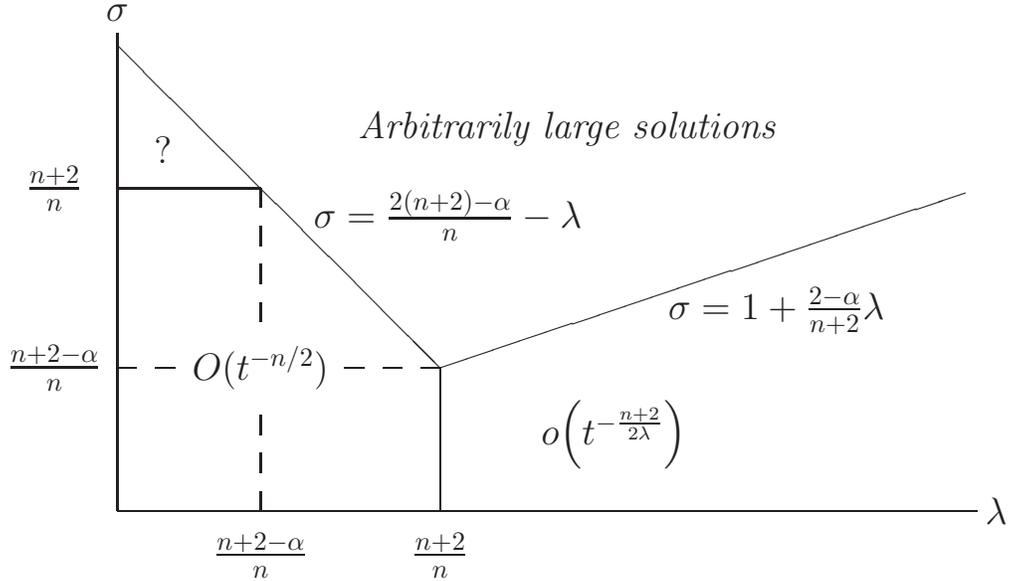

\begin{thm}\label{thm1.1}
  Suppose $u$ is a nonnegative solution of \eqref{1.1},\eqref{1.2} where
  $\alpha\in(0,n+2)$, $\lambda>0$, $T>0$, and
  $$0\leq\sigma<g_\alpha (\lambda)$$  
 are constants and $\Omega$ is an open subset of $\R^n$. Then for each
 compact subset $K$ of $\Omega$ we have as $t\to 0^+$ that 
\begin{numcases}{\max_{x\in K}u(x,t)=}
O(t^{-n/2}) & if $0<\lambda<\frac{n+2}{n}$ \label{1.6}\\
o(t^{-(n+2)/(2\lambda)}) & if $\lambda\ge\frac{n+2}{n}.$ \label{1.7}
\end{numcases}
\end{thm}

The estimate \eqref{1.6} is optimal by Remark \ref{rem1}.
The exponent $-(n+2)/(2\lambda)$ in \eqref{1.7} is also optimal by the
following result.

\begin{thm}\label{thm1.2}
Suppose
\[
\lambda\ge\frac{n+2}{n}\quad\text{and}\quad 
\gamma=\frac{n+2-\varepsilon}{2\lambda}
\]
for some $\varepsilon\in(0,1)$. Then there exists a $C^\infty$
positive solution $u$ of 
\[
Hu=0\quad\text{in }\R^n\times (0,\infty)
\]
such that
\[
u\in L^\lambda(\R^n\times(0,T))\quad\text{for all } T>0
\]
and
\[
u(0,t)=t^{-\gamma}\quad\text{for all } t>0.
\]
\end{thm}

By the next theorem, if the point $(\lambda,\sigma)$ lies above the
graph of $\sigma=G_\alpha(\lambda)$ then there does not exist a
pointwise bound for nonnegative solutions $u$ of
\eqref{1.1},\eqref{1.2} on compact subsets of $\Omega$ as $t\to 0^+$.

\begin{thm}\label{thm1.3}
 Suppose $\alpha,\lambda$, and $\sigma$ are constants satisfying
 $$\alpha\in(0,n+2), \quad \lambda>0, \quad \text{and}\quad\ 
\sigma>G_\alpha (\lambda).$$
 Let $\varphi:(0,1)\to(0,\infty)$ be a continuous function satisfying
 $$\lim_{t\to0^+}\varphi(t)=\infty.$$
 Then there exists a positive solution $u$ of \eqref{1.1},\eqref{1.2}
 with $T=1$ and $\Omega=\R^n$ such that
 $$u(0,t)\neq O(\varphi(t))\quad\text{ as }t\to 0^+.$$
\end{thm}

Theorems \ref{thm1.1}--\ref{thm1.3} completely answer Question \ref{ques1} when
the point $(\lambda,\sigma)$ lies below the graph of $g_\alpha$ or
above the graph of $G_\alpha$.  In particular, if $u$ is a nonnegative
solution of \eqref{1.1},\eqref{1.2} where $(\lambda,\sigma)$ lies in the
first quadrant of the $\lambda\sigma$-plane and either $\sigma<
g_\alpha(\lambda)$ or $\sigma> G_\alpha(\lambda)$ then according to
Theorems \ref{thm1.1}--\ref{thm1.3} either
\begin{enumerate}
\item[(i)] $\varphi(t)=t^{-n/2}$ is an optimal a priori pointwise bound for
  $u$ on compact subsets of $\Omega$ as $t\to 0^+$; or
\item[(ii)] $\varphi(t)=t^{-(n+2)/(2\lambda)}$ is an optimal a priori
  pointwise bound for $u$ on compact subsets of $\Omega$ as $t\to
  0^+$; or
\item[(iii)] no pointwise a priori bound exists for $u$ on compact
  subsets of $\Omega$ as $t\to 0^+$, that is solutions can be
  arbitrarily large as $t\to 0^+$.
\end{enumerate}
The regions in which these three possibilities occur are shown in
Figures~\ref{fig1} and \ref{fig2}. Also included in
Figures~\ref{fig1} and \ref{fig2} is an open triangular
region marked with a question mark. For $(\lambda,\sigma)$ in this
region we have no results for Question \ref{ques1}.

Concerning the case that
$(\lambda,\sigma)$ lies on the graph of $g_\alpha$ we have the
following result.

\begin{thm}\label{thm1.4}
 Suppose $\alpha\in(0,n+2)$.
 \begin{enumerate}
 \item[(i)] If $0<\lambda<\frac{n+2-\alpha}{n}$ and $\sigma=g_\alpha
   (\lambda)$ then $\varphi(t)=t^{-n/2}$ is a poinwise bound for
     nonnegative solutions $u$ of \eqref{1.1},\eqref{1.2} on compact
     subsets of $\Omega$ as $t\to 0^+$.
   \item[(ii)] If $\alpha\in(2,n+2)$, $\lambda>\frac{n+2}{\alpha-2}$,
     and $\sigma=g_\alpha(\lambda)$ then there does not exist an a
     priori pointwise bound for nonnegative solutions $u$ of
     \eqref{1.1},\eqref{1.2} on compact subsets of $\Omega$ as $t\to
     0^+$.
 \end{enumerate}
\end{thm}

When a pointwise a priori bound as $t\to 0^+$ for nonnegative solutions
$u$ of \eqref{1.1},\eqref{1.2} on compact subsets of $\Omega$ does not
exist, as in Theorems \ref{thm1.3} and \ref{thm1.4}(ii), we prove this
by constructing for any given continuous function
$\varphi:(0,1)\to(0,\infty)$ a nonnegative solution $u$ of
\eqref{1.1},\eqref{1.2} consisting of a sequence of smoothly connected
peaks centered at $(x_j,t_j)$ where $t_j\to 0^+$ such that
\[
u(x_j,t_j)\not= O(\varphi(t_j)) \quad\text{as }j\to\infty.
\]
When such a pointwise a priori bound does exist, as in Theorems
\ref{thm1.1} and \ref{thm1.4}(i), we reduce the proof of this fact to
ruling out the possibility of such peaked solutions.

 If $\alpha\in (0,n+2)$ and $\lambda>0$ then one of the following
 three conditions holds:
\begin{enumerate}
 \item[(i)] $0<\lambda<\frac{n+2-\alpha}{n}$; 
 \item[(ii)] $\frac{n+2-\alpha}{n}\leq\lambda<\frac{n+2}{n}$;
 \item[(iii)] $\frac{n+2}{n}\leq\lambda<\infty$.
\end{enumerate}

The proofs of Theorems \ref{thm1.1}--\ref{thm1.4} in case (i)
(resp.~(ii), (iii)) are given in Section \ref{sec3}
(resp.~\ref{sec4}, \ref{sec5}).  In Section \ref{sec2} we provide some
lemmas needed for these proofs. Our approach relies on an integral
representation formula for nonnegative supertemperatures (see Appendix
\ref{secA}), some integral estimates for heat potentials (see
Appendix \ref{secB}), and Moser's iteration (see Lemmas \ref{lem4.1}
and \ref{lem5.2}).

In this paper, we denote by $\mathcal{P}_r(x,t)$ the open circular
cylinder in $\R^n\times\R$ of radius $\sqrt{r}$, height $r$, and top
center point $(x,t)$. Thus
\[
\mathcal{P}_r(x,t)=\{(y,s)\in \R^n\times\R:|y-x|<\sqrt{r}\text{ and }t-r<s<t\}.
\]

%%%%%%%%%%%%%%%%%%%%%%%%%%%%%%%%%%%%%%%%%%%%%%%

\begin{rem}\label{rem2}
  Note that 
\begin{equation}\label{1.9}
\Phi(x,t)^{\alpha/n}=\frac{1}{(4\pi)^{\alpha/2}}t^{-\alpha/2}
e^{-\frac{\alpha}{4n}\frac{|x|^2}{t}}\chi_{(0,\infty)}(t)\quad\text{in
}\R^n\times\R.
\end{equation}
However, by checking the proofs of our results, we find that Theorems
\ref{thm1.1}, \ref{thm1.3}, and \ref{thm1.4} remain correct if
$\Phi(x,t)^{\alpha/n}$ in \eqref{1.2} is replaced with any function of
the form
\begin{equation}\label{1.10}
C_1(n,\alpha)t^{-\alpha/2}
e^{-C_2(n,\alpha)\frac{|x|^2}{t}}\chi_{(0,\infty)}(t)\quad\text{in
}\R^n\times\R,
\end{equation}
where $C_1(n,\alpha)$ and $C_2(n,\alpha)$ are any given positive
constants.  
In particular, since the fundamental solution $\Phi_\alpha$
of the fractional heat operator
$(\frac{\partial}{\partial t}-\Delta)^{\alpha/2}$, $\alpha\in(0,n+2)$, is given by
\[
\Phi_\alpha(x,t):=\frac{t^{\alpha/2-1}}{\Gamma(\alpha/2)}\Phi(x,t),
\]
where $\Phi$ is the heat kernel \eqref{1.3} (see \cite[Chapter 9,
Section 2]{S2002}), we find for $0<\alpha<n+2$ that
\[
\Phi_{n+2-\alpha}(x,t)=\frac{1}{(4\pi)^{n/2}\Gamma((n+2-\alpha)/2)}t^{-\alpha/2}
e^{-\frac{1}{4}\frac{|x|^2}{t}}\chi_{(0,\infty)}(t)
\]
is of the form \eqref{1.10}. Thus Theorems \ref{thm1.1}, \ref{thm1.3},
and \ref{thm1.4} remain correct if $\Phi^{\alpha/n}$ in \eqref{1.2} is replaced
with $\Phi_{n+2-\alpha}$.
\end{rem}

\section{Preliminary Lemmas}\label{sec2}

\begin{lem}\label{lem2.1}
  Suppose $\alpha\in(0,n+2),\,\lambda>0,\,\sigma\geq0,\,T>0$, and
  $\beta\geq0$ are constants, $\Omega$ is an open subset of
  $\mathbb{R}^n$, and $K$ is a compact subset of $\Omega$, such that
  there exists a nonnegative solution $u$ of \eqref{1.1},\eqref{1.2},
  where the convolution operation in \eqref{1.2} is in
  $\mathbb{R}^n \times(0,T)$, satisfying
 \begin{equation}\label{2.1}
  \max_{x\in K}u(x,t)\neq O(t^{-\beta}),\,\quad (\text{resp. }o(t^{-\beta}))\quad\text{as }t\to0^+ . 
 \end{equation}
 Then there exists a nonnegative function $v(\xi,\tau)$ such that
 \begin{equation}\label{2.2}
   v\in C^{2,1}(\mathbb{R}^n \times(0,16))
\cap L^\lambda (\mathbb{R}^n \times(0,16)),
 \end{equation}
 \begin{equation}\label{2.3}
  0\leq Hv\leq(\Phi^{\alpha/n}*v^\lambda )v^\sigma \quad\text{in }B_4 (0)\times(0,16),
 \end{equation}
 where $*$ is the convolution operation in $\mathbb{R}^n \times(0,16)$, and 
 \begin{equation}\label{2.4}
  \max_{|\xi|\leq1}v(\xi,\tau)\neq O(\tau^{-\beta}),
\quad (\text{resp. }o(\tau^{-\beta}))\quad\text{as }\tau\to 0^+ .
 \end{equation}
\end{lem}

\begin{proof}
 It follows from \eqref{2.1} and the compactness of $K$ that there
 exist a sequence $\{(x_j ,t_j )\}\subset K\times(0,T)$ and $x_0 \in K$ such that
 \begin{equation}\label{2.5}
  (x_j ,t_j )\to(x_0 ,0)\quad\text{as }j\to\infty
 \end{equation}
 and
 \begin{equation}\label{2.6}
  u(x_j ,t_j )\neq O(t^{-\beta}_{j})\quad (\text{resp. }o(t^{-\beta}_{j}))\quad\text{as }j\to\infty .
 \end{equation}
 Choose $r\in(0,1)$ and $b>0$ such that
 \begin{equation}\label{2.7}
  \overline{B_{4r}(x_0 )}\times(0,16r^2 )\subset\Omega\times (0,T)
 \end{equation}
 and
 \begin{equation}\label{2.8}
  b^{\lambda+\sigma-1}<r^{-(n+4-\alpha)}.
 \end{equation}
 Define $v(\xi,\tau)$ by $u(x,t)=bv(\xi,\tau)$ where $x=x_0 +r\xi$ and $t=r^2 \tau$ and
 define $(\xi_j ,\tau_j )$ by $x_j =x_0 +r\xi_j$ and
 $t_j =r^2 \tau_j$.  Then by \eqref{2.5}
 \begin{equation}\label{2.9}
  (\xi_j ,\tau_j )\to(0,0)\quad\text{as }j\to\infty.
 \end{equation}
 Clearly $(x,t)\in\mathbb{R}^n \times(0,16r^2 )$ if and only if $(\xi,\tau)\in\mathbb{R}^n \times(0,16)$.  Also $16r^2 \leq T$ by \eqref{2.7}.  It therefore follows from \eqref{1.1} that \eqref{2.2} holds.
 
 For $(x,t)\in \mathcal{P}_{16r^2}(x_0 ,16r^2 )$ (i.e. $(\xi,\tau)\in
 \mathcal{P}_{16}(0,16)$) we have under the change of variables $y=x_0 +r\eta,\,s=r^2 \zeta$ that 
 $$\iint_{\mathbb{R}^n \times(0,T)}\Phi(x-y,t-s)^{\alpha/n}u(y,s)^\lambda\,dy\,ds=\frac{b^\lambda r^{n+2}}{r^\alpha}\iint_{\mathbb{R}^n \times(0,16)}\Phi(\xi-\eta,\tau-\zeta)^{\alpha/n}v(\eta,\zeta)^\lambda \,d\eta\,d\zeta$$
 where in the last integral we were able to replace the region of integration $\mathbb{R}^n \times(0,T/r^2 )$ with $\mathbb{R}^n\times(0,16)$ because $\tau<16\leq T/r^2$ and $\Phi(x,t)=0$ 
 for $t<0$.  Thus by \eqref{1.2} and \eqref{2.8} we find that $v$ satisfies \eqref{2.3}.
 
 Finally by \eqref{2.6} we have 
 $$\tau^{\beta}_{j}v(\xi_j ,\tau_j
 )=\left(\frac{t_j}{r^2}\right)^\beta \frac{1}{b}u(x_j ,t_j
 )\neq O(1)\quad (\text{resp. } o(1))\quad\text{as }j\to\infty$$
 which together with \eqref{2.9} implies \eqref{2.4}.
\end{proof}

\begin{rem}\label{rem2.1}
  Suppose $\alpha$, $\lambda$, $\sigma$, $T$, $\beta$, $\Omega$,
  and $K$ are as in Lemma \ref{lem2.1}.  Then in order to show that
  all nonnegative solutions $u$ of \eqref{1.1},\eqref{1.2} satisfy
 $$\max_{x\in K}u(x,t)=O(t^{-\beta})\quad (\text{resp. }o(t^{-\beta}))\quad\text{as }t\to 0^+$$
 it suffices by Lemma \ref{lem2.1} to show that all nonnegative solutions $u(x,t)$ of
 \begin{equation}\label{2.10}
  u\in C^{2,1}(\mathbb{R}^n \times(0,16))\cap L^\lambda (\mathbb{R}^n \times(0,16))
 \end{equation}
 and
 \begin{equation}\label{2.11}
  0\leq Hu\leq(\Phi^{\alpha/n}*u^\lambda )u^\sigma \quad\text{in }B_4 (0)\times(0,16),
 \end{equation}
 where $*$ is the convolution operation in $\mathbb{R}^n \times(0,16)$, satisfy
 $$\max_{|x|\leq1}u(x,t)=O(t^{-\beta})\quad (\text{resp. }o(t^{-\beta}))\quad\text{as }t\to 0^+.$$
\end{rem}

Throughout this paper we will repeatedly use the following simple lemma.

\begin{lem}\label{lem2.2}
 If $\gamma>0$ and $x\in\mathbb{R}^n ,\,n\geq1$, then
 $$\int_{|x-y|<r}e^{-\gamma|x-y|^2}dy=\gamma^{-n/2}\int_{|z|<\sqrt{\gamma}r}e^{-|z|^2}dz$$
 and
 $$\int_{|x-y|>r}e^{-\gamma|x-y|^2}dy=\gamma^{-n/2}\int_{|z|>\sqrt{\gamma}r}e^{-|z|^2}dz.$$
 In particular
 $$\int_{\mathbb{R}^n}e^{-\gamma|x-y|^2}dy=C(n)\gamma^{-n/2}.$$
\end{lem}

\begin{proof}
 Make the change of variables $z=\sqrt{\gamma}(x-y)$.
\end{proof}

\begin{lem}\label{lem2.3}
 Suppose for some constants $\alpha\in(0,n+2),\,\lambda>0$, and $\sigma\geq0$, the function $u$ is a nonnegative solution of \eqref{2.10},\eqref{2.11} where $*$ is the convolution operation in $\mathbb{R}^n \times(0,16)$.  Set $v=u+1$.  Then
 \begin{equation}\label{2.12}
  v\in C^{2,1}(\mathbb{R}^n \times(0,16))\cap L^\lambda (B_{\sqrt{8}}(0)\times(0,8))
 \end{equation}
 and for some positive constant $C$, $v$ satisfies
 \begin{equation}\label{2.13}
  \begin{rcases}
   0\leq Hv\leq C(\Phi^{\alpha/n}*v^\lambda )v^\sigma\\
   v\geq1
  \end{rcases} \quad\text{in }B_2 (0)\times(0,8)
 \end{equation}
 where $*$ is the convolution operation in $B_{\sqrt{8}}(0)\times(0,8)$.  Also
 \begin{equation}\label{2.14}
  Hv,\,v^\beta \in L^1 (B_{\sqrt{8}}(0)\times(0,8))\quad
\text{for all }\beta\in\left[1,\frac{n+2}{n}\right)
 \end{equation}
 and there exists a positive finite Borel measure $\mu$ on
 $B_{\sqrt{8}}(0)$ and a bounded function\\
 $h\in C^{2,1}(B_2 (0)\times(-4,4))$ satisfying
 \begin{align*}
  H&h=0\quad\text{in }B_2 (0)\times(-4,4)\\
  &h=0\quad\text{in }B_2 (0)\times(-4,0]
 \end{align*}
 such that
 \begin{equation}\label{2.15}
  v(x,t)=h(x,t)+\int^{8}_{0}\int_{|y|<\sqrt{8}}\Phi(x-y,t-s)Hv(y,s)\,dy\,ds+\int_{|y|<\sqrt{8}}\Phi(x-y,t)\,d\mu(y)
 \end{equation}
for $(x,t)\in B_2 (0)\times(0,4)$.
\end{lem}

\begin{rem}\label{rem2.2}
 Under the assumptions of Lemma \ref{lem2.3} we have
 $$(4\pi t)^{n/2}\int_{|y|<\sqrt{8}}\Phi(x-y,t)\,d\mu(y)\leq\int_{|y|<\sqrt{8}}d\mu(y)<\infty\quad\text{for }(x,t)\in\mathbb{R}^n \times(0,\infty).$$
 Thus by \eqref{2.15} we see that
 $$v(x,t)\leq C\left(\left(\frac{1}{\sqrt{t}}\right)^n +\int^{8}_{0}\int_{|y|<\sqrt{8}}\Phi(x-y,t-s)Hv(y,s)\,dy\,ds\right)\quad\text{for }(x,t)\in B_2 (0)\times(0,4).$$
\end{rem}

\begin{proof}[Proof of Lemma \ref{lem2.3}.] Clearly \eqref{2.10} implies \eqref{2.12}.  For
$$(x,t)\in B_2 (0)\times(0,8), \quad (y,s)\in(\mathbb{R}^n \times(0,16))\backslash(B_{\sqrt{8}}(0)\times(0,8)),$$
and $s<t$ we have $|x-y|>\sqrt{8}-2>1/\sqrt{2}$ and thus
\begin{align*}
 \Phi(x-y,t-s)&\leq\frac{1}{(4\pi(t-s))^{n/2}}e^{-\frac{1}{8(t-s)}}\\
 &\leq\sup_{0<\tau<t}\frac{e^{-\frac{1}{8\tau}}}{(4\pi\tau)^{n/2}}\leq C\frac{e^{-\frac{1}{8t}}}{t^{n/2}}.
\end{align*}
Hence for $(x,t)\in B_2 (0)\times(0,8)$ we have
\begin{align}\label{2.16}
 \notag \iint_{\mathbb{R}^n \times(0,16)\backslash
  B_{\sqrt{8}}(0)\times(0,8)}&
\Phi(x-y,t-s)^{\alpha/n}u(y,s)^\lambda
  \,dy\,ds\\
\leq C&\left(\frac{e^{-\frac{1}{8t}}}{t^{n/2}}\right)^{\alpha/n}\iint_{\mathbb{R}^n \times(0,16)}u(y,s)^\lambda \,dy\,ds
 \leq C\left(\frac{e^{-\frac{1}{8t}}}{t^{n/2}}\right)^{\alpha/n}
\end{align}
by \eqref{2.10}.

On the other hand, for $(x,t)\in B_2 (0)\times(0,8)$ we have $B_R (x)\subset B_{\sqrt{8}}(0)$ where $R=\sqrt{8}-2$ and thus by Lemma \ref{lem2.2} we find that
\begin{align}
 \notag
  \iint_{B_{\sqrt{8}}(0)\times(0,8)}\Phi(x-y,t-s)^{\alpha/n}\,dy\,ds
&\geq\int^{t}_{0}\left(\int_{B_R (x)}\Phi(y-x,t-s)^{\alpha/n}dy\right)ds\\
 \notag &=\int^{t}_{0}\int_{B_R
          (x)}\Phi(y-x,\tau)^{\alpha/n}dy\,d\tau\\
\notag &=\int_0^t\frac{1}{(4\pi\tau)^{\alpha/2}}
\left(\int_{|y-x|<R}e^{-\frac{\alpha}{4n\tau}|y-x|^2}dy\right)d\tau\\
\notag &=C\int^{t}_{0}\tau^{\frac{n-\alpha}{2}}\left(\int_{|z|<R\sqrt{\frac{\alpha}{4n\tau}}}e^{-|z|^2}dz\right)d\tau\\
 \label{2.17}&\geq C\int^{t}_{0}\tau^{\frac{n-\alpha}{2}}d\tau=Ct^{\frac{n+2-\alpha}{2}}
 \geq C\left(\frac{e^{-\frac{1}{8t}}}{t^{n/2}}\right)^{\alpha/n}.
\end{align}
Hence for $(x,t)\in B_2 (0)\times(0,8)$ we obtain from \eqref{2.16} and \eqref{2.17} that
\begin{align*}
 &\iint_{\mathbb{R}^n \times(0,16)}\Phi(x-y,t-s)^{\alpha/n}u(y,s)^\lambda \,dy\,ds\\
 &\leq\iint_{B_{\sqrt{8}}(0)\times(0,8)}\Phi(x-y,t-s)^{\alpha/n}u(y,s)^\lambda
   \,dy\,ds
+C\iint_{B_{\sqrt{8}}(0)\times(0,8)}\Phi(x-y,t-s)^{\alpha/n}1^\lambda \,dy\,ds\\
 &\leq C \iint_{B_{\sqrt{8}}(0)\times(0,8)}\Phi(x-y,t-s)^{\alpha/n}v(y,s)^\lambda \,dy\,ds.
\end{align*}
Thus, since $u$ satisfies \eqref{2.11} we see that $v$ satisfies
\eqref{2.13}.  Finally, by \eqref{2.11}, $Hv\ge 0$ in
$B_4 (0)\times(0,16)$.  Hence Theorem \ref{thmA} and Remark \ref{remA}
with $R_1=4$, $R_2=8$, and $R_3=16$ imply \eqref{2.14} and
\eqref{2.15}.
\end{proof}

The following lemma will be needed to estimate the last integral in
\eqref{2.15}.

\begin{lem}\label{lem2.4}
 Suppose
 \begin{equation}\label{2.18}
  u\in L^p (\Omega \times(0,T))
 \end{equation}
 for some open subset $\Omega$ of $\mathbb{R}^n ,\,n\geq1$, and some constants $p\in[1,\infty)$ and $T>0$.  Assume also that 
 $$u(x,t)=\int_{\mathbb{R}^n}\Phi(x-y,t)\, d\mu(y)$$
 for some finite positve Borel measure $\mu$ on $\mathbb{R}^n$.  Then
 for each compact subset $K$ of $\Omega$ we have
 \begin{equation}\label{2.19}
  \max_{x\in K}u(x,t)=o\left(t^{-\frac{n+2}{2p}}\right)\quad\text{as }t\to 0^+ .
 \end{equation}
\end{lem}

\begin{proof}
 The proof consists of two steps.\\
 \textbf{Step 1.} In this step we prove Lemma \ref{lem2.4} in the special case that
 \begin{equation}\label{2.20}
  \Omega=B_{3r}(x_0 )\quad\text{and}\quad K=\overline{B_r (x_0 )}
 \end{equation}
 for some $x_0 \in\mathbb{R}^n$ and some $r>0$.  Clearly we can assume
 $x_0 =0$.  Since $u=v+w$ where
 $$v(x,t)=\int_{|y|<2r}\Phi(x-y,t)\,d\mu(y)$$
 and
 $$w(x,t)=\int_{|y|\geq2r}\Phi(x-y,t)\,d\mu(y),$$
 to complete step 1, it suffices to prove $v$ and $w$ satisfy \eqref{2.19} when $\Omega$ and $K$ are given by \eqref{2.20}.
 
 Since for $|x-y|\geq r$ and $t>0$
 $$\Phi(x-y,t)\leq\frac{1}{r^n}\left(\frac{r^2}{4\pi t}\right)^{n/2}e^{-\frac{r^2}{4t}}\leq r^{-n}C(n)$$
 we have
 $$\max_{|x|\leq r}w(x,t)\leq r^{-n}C(n)\int_{\mathbb{R}^n}d\mu(y)<\infty\quad\text{for }t>0.$$
 Thus $w$ satisfies \eqref{2.19} when $\Omega$ and $K$ are given by \eqref{2.20}.
 
 For $|y|\leq 2r$ and $\tau>0$ it follows from Lemma \ref{lem2.2} that
 \begin{align*}
  \int_{|x|\geq3r}\Phi(x-y,\tau)^p dx&=\frac{1}{(4\pi\tau)^{np/2}}\int_{|x|\geq3r}e^{-\frac{p|x-y|^2}{4\tau}}dx\\
  &\leq\frac{1}{(4\pi\tau)^{np/2}}\int_{|x-y|\geq r}e^{-\frac{p|x-y|^2}{4\tau}}dx\\
  &=\frac{C(n,p)}{r^{n(p-1)}}\left[\left(\frac{r}{\sqrt{\tau}}\right)^{n(p-1)}\int_{|z|>\sqrt{\frac{p}{4}}\frac{r}{\sqrt{\tau}}}e^{-|z|^2}dz\right]\\
  &\leq C(n,p)/r^{n(p-1)}.
 \end{align*}
 We obtain therefore from Jensen's inequality and Fubini's theorem that
 \begin{align}\label{2.21}
  \notag \| v\|^{p}_{L^p ((\mathbb{R}^n \backslash
           B_{3r}(0))\times(0,t))}
&=\int^{t}_{0}\int_{|x|\ge 3r}\left(\int_{|y|<2r}\Phi(x-y,\tau)\,d\mu(y)\right)^p dx\,d\tau\\
  \notag &\le |\mu|^{p-1}\int_{|y|<2r}\int^{t}_{0}\left(\int_{|x|\geq3r}\Phi(x-y,\tau)^p dx\right)d\tau \,d\mu(y)\\
  &\leq|\mu|^p C(n,p)t/r^{n(p-1)}\quad\text{for all }t>0.
 \end{align}
 We now use \eqref{2.21} to show $v$ satisfies \eqref{2.19}.
 
 For $0<\tau<t$ and $x\in\mathbb{R}^n$ it follows from standard
 $L^p$-$L^q$ estimates with $q=\infty$ (see \cite[Prop. 48.4]{QS} that
 $$v(x,t)\leq(4\pi)^{\frac{-n}{2p}}(t-\tau)^{\frac{-n}{2p}}\| v(\cdot,\tau)\|_{L^p (\mathbb{R}^n )}.$$
 Hence
 $$v(x,t)^p \int^{t}_{0}(t-\tau)^{n/2}d\tau\leq(4\pi)^{-n/2}\| v\|_{L^p (\mathbb{R}^n \times(0,t))}^p$$
 which implies
 \begin{align*}
  \max_{x\in\mathbb{R}^n}v(x,t)t^{\frac{n+2}{2p}}&\leq C(n,p)\| v\|_{L^p (\mathbb{R}^n \times(0,t))}\\
  &\leq C(n,p)\left[\| u\|_{L^p (B_{3r}(0)\times(0,t))}+\| v\|_{L^p (\mathbb{R}^n \backslash B_{3r}(0)\times(0,t))}\right]\\
  &\to0\quad\text{as }t\to0^+
 \end{align*}
 by \eqref{2.18} and \eqref{2.21}.  Thus $v$ satisfies \eqref{2.19} when $\Omega$ and $K$ are given by \eqref{2.20}.
\medskip
 
 \noindent\textbf{Step 2.} We now use Step 1 to complete the proof.  For each $x\in K$ choose $r_x >0$ such that $B_{3r_x}(x)\subset\Omega$.  Since $K$ is compact there exists finitely many points $x_1 ,...,x_m$ in $K$ such that
 \begin{equation}\label{2.22}
  K\subset\bigcup^{m}_{j=1}B_{r_j}(x_i )\quad \text{where }r_j =r_{x_j}.
 \end{equation}
 For $j=1,2,...,m$ we have by Step 1 that
 $$\max_{|x-x_j |\leq r_j}u(x,t)=o\left(t^{-\frac{n+2}{2p}}\right)\quad\text{as }t\to0^+ .$$
 Hence \eqref{2.19} follows from \eqref{2.22}.
\end{proof}

\begin{lem}\label{lem2.5}
 Suppose $r>0$ and $\beta>n+2$ are constants and $(x_0 ,t_0 )\in\mathbb{R}^n \times\mathbb{R}$.  Then for $(x,t)\in\overline{\mathcal{P}_r (x_0 ,t_0 )}$ we have
 $$\iint_{\mathbb{R}^n \times\mathbb{R}\backslash\mathcal{P}_{2r}(x_0 ,t_0 )}\Phi(x-y,t-s)^{\beta/n}\,dy\,ds\leq\frac{C}{\sqrt{r}^{\beta-(n+2)}}$$
 where $C=C(n,\beta)>0$.
\end{lem}

\begin{proof}
 Throughout this proof $(x,t)\in\overline{\mathcal{P}_r (x_0 ,t_0 )}$ and $C=C(n,\beta)$ is a positive constant whose value may change from line to line.  Let
 $$A=\mathbb{R}^n \times(-\infty,t_0 -2r]\quad\text{and}\quad B=(\mathbb{R}^n \backslash B_{\sqrt{2r}}(x_0 ))\times(t_0 -2r,\,t_0 ).$$
 For $(y,s)\in B$ we have
 $$\frac{|y-x|}{|y-x_0 |}\geq\frac{|y-x_0 |-|x-x_0 |}{|y-x_0 |}=1-\frac{|x-x_0 |}{|y-x_0 |}\geq1-\frac{\sqrt{r}}{\sqrt{2r}}=1-\frac{1}{\sqrt{2}}>\frac{1}{4}.$$
 It therefore follows from Lemma \ref{lem2.2} that
 \begin{align*}
  \iint_B \Phi(x-y,t-s)^{\beta/n}\,dy\,ds&\leq\int^{t}_{t_0 -2r}\frac{1}{(4\pi(t-s))^{\beta/2}}\left(\int_{|y-x_0 |>\sqrt{2r}}e^{-\frac{\beta|y-x_0 |^2}{64n(t-s)}}dy\right)ds\\
  &=Cr^{\frac{n-\beta}{2}}\int^{t}_{t_0 -2r}\left(\frac{r}{t-s}\right)^{\frac{\beta-n}{2}}\int_{|z|>\sqrt{\frac{2\beta}{64n}}\sqrt{\frac{r}{t-s}}}e^{-|z|^2}dz\,ds\\
  &\leq Cr^{\frac{n+2-\beta}{2}}.
 \end{align*}
 Also, by Lemma \ref{lem2.2}, we obtain
 \begin{align*}
  \iint_{A}\Phi(x-y,t-s)^{\beta/n}\,dy\,ds&=\int^{t_0 -2r}_{-\infty}\frac{1}{(4\pi(t-s))^{\beta/2}}\int_{\mathbb{R}^n}e^{-\frac{\beta|x-y|^2}{4n(t-s)}}\,dy\,ds\\
  &=C\int^{t_0 -2r}_{-\infty}(t-s)^{\frac{n-\beta}{2}}ds\\
  &=C(t-t_0 +2r)^{\frac{n+2-\beta}{2}}\leq Cr^{\frac{n+2-\beta}{2}}.
 \end{align*}
 Thus Lemma \ref{lem2.5} follows from the fact that
 $$\iint_{\mathbb{R}^n \times\mathbb{R}\backslash\mathcal{P}_{2r}(x_0
   ,t_0 )}\Phi(x-y,t-s)^{\beta/n}\,dy\,ds=\iint_{A\cup B} \Phi(x-y,t-s)^{\beta/n}\,dy\,ds.$$
\end{proof}

\begin{lem}\label{lem2.6}
 Suppose $r>0$ and $0<\beta<n+2$ are constants and $(x_0 ,t_0 )\in\mathbb{R}^n \times\mathbb{R}$.  Then
 $$\iint_{(y,s)\in\mathcal{P}_r (x_0 ,t_0 )}\Phi(x-y,t-s)^{\beta/n}\,dy\,ds\leq C\sqrt{r}^{n+2-\beta}\quad\text{for }(x,t)\in\mathcal{P}_r (x_0 ,t_0 )$$
 where $C=C(n,\beta)>0$.
\end{lem}

\begin{proof}
 By Lemma \ref{lem2.2}, we have for $(x,t)\in\mathcal{P}_r (x_0 ,t_0 )$ that
 \begin{align*}
  \iint_{\mathcal{P}_r (x_0 ,t_0 )}\Phi(x-y,t-s)^{\beta/n}\,dy\,ds&\leq\int^{t}_{t_0 -r}\frac{1}{(4\pi(t-s))^{\beta/2}}\left(\int_{\mathbb{R}^n}e^{-\frac{\beta|x-y|^2}{4n(t-s)}}dy\right)ds\\
  &=C\int^{t}_{t_0 -r}(t-s)^{\frac{n-\beta}{2}}ds\\
  &=C(t-t_0 +r)^{\frac{n+2-\beta}{2}}\leq C\sqrt{r}^{n+2-\beta}.
 \end{align*}
\end{proof}

\begin{lem}\label{lem2.7}
 Suppose $\alpha\in(0,n+2)$ and $\beta\in[0,n+2)$ are constants.  Then
 \begin{equation}\label{2.23}
  \iint_{\mathbb{R}^n \times (0,t)}\Phi(x-y,t-s)^{\alpha/n}\Phi(y-z,s)^{\beta/n}\,dy\,ds\leq\frac{C}{\sqrt{t}^{\alpha+\beta-(n+2)}}
 \end{equation}
 for all $x,z\in\mathbb{R}^n$ and $t>0$ where $C=C(n,\alpha,\beta)>0$.
\end{lem}

\begin{proof} When $\beta =0$, Lemma \ref{lem2.7} follows directly
  from Lemma \ref{lem2.2}. Hence we can assume $\beta\in (0,n+2)$.
 Under the change of variables
 \[
x-z=\sqrt{t}\xi,\quad y-z=\sqrt{t}\eta,\quad s=t\zeta 
 \]
 we see that the left side of \eqref{2.23} equals
 \begin{align*}
  &\iint_{\mathbb{R}^n \times(0,1)}\Phi(\sqrt{t}(\xi-\eta),t(1-\zeta))^{\alpha/n}\Phi(\sqrt{t}\eta,t\zeta)^{\beta/n}\sqrt{t}^{n+2}\,d\eta\,d\zeta\\
  &=\iint_{\mathbb{R}^n \times(0,1)}\left(\frac{1}{(4\pi
    t(1-\zeta))^{n/2}}\right)^{\alpha/n}\left(\frac{1}{(4\pi
    t\zeta)^{n/2}}\right)^{\beta/n}
e^{-\frac{\alpha}{n}\frac{|\xi-\eta|^2}{4(1-\zeta)}-\frac{\beta}{n}\frac{|\eta|^2}{4\zeta}}
\sqrt{t}^{n+2}\,d\eta\,d\zeta\\
  &=\frac{C(n,\alpha,\beta)}{\sqrt{t}^{\alpha+\beta-(n+2)}}\int^{1}_{0}\frac{1}{(1-\zeta)^{\alpha/2}\zeta^{\beta/2}}\left(\int_{\mathbb{R}^n}
e^{-\frac{\alpha}{n}\frac{|\xi-\eta|^2}{4(1-\zeta)}-\frac{\beta}{n}\frac{|\eta|^2}{4\zeta}}d\eta\right)d\zeta\\
  &\leq\frac{C(n,\alpha,\beta)}{\sqrt{t}^{\alpha+\beta-(n+2)}}\left[\int^{1/2}_{0}\frac{d\zeta}{\zeta^{\beta/2-n/2}}+\int^{1}_{1/2}\frac{d\zeta}{(1-\zeta)^{\alpha/2-n/2}}\right]
 \end{align*}
by Lemma \ref{lem2.2}.
\end{proof}

\begin{lem}\label{lem2.8}
Suppose $(x_0,t_0)\in\R^n\times\R$ and $r>0$. If
  \[
(x,t)\in\overline{\mathcal{P}_r(x_0,t_0)} \quad \text{and}\quad 
  (y,s)\in(\R^n\times\R)\setminus\overline{\mathcal{P}_{2r}(x_0,t_0)}
\] then
\[
\Phi(x-y,t-s)\le\frac{C(n)}{r^{n/2}}.
\] 
\end{lem}

\begin{proof} We consider three cases.

\noindent {\bf Case I}. Suppose $t_0-2r\le s<t$. Then $|x-y|\ge
  (\sqrt{2}-1) \sqrt{r}$ and hence 
\begin{align*}
\Phi(x-y,t-s)&\le\frac{e^{-\frac{(\sqrt{2}-1)^2r}{4(t-s)}}}{(4\pi
  (t-s))^{n/2}}\le\sup_{\tau>0}\frac{e^{-\frac{(\sqrt{2}-1)^2r}{4\tau}}}{(4\pi
  \tau)^{n/2}}\\
             &=\sup_{\zeta>0}\frac{e^{-(\sqrt{2}-1)^2\zeta}}{(\pi r/\zeta)^{n/2}}
=\frac{C(n)}{r^{n/2}}.
\end{align*}
\noindent {\bf Case II}. Suppose $s<t_0-2r$. Then $t-s\ge r$ and hence 
\[
\Phi(x-y,t-s)\le\frac{1}{(4\pi r)^{n/2}}=\frac{C(n)}{r^{n/2}}.
\]
\noindent {\bf Case III}. Suppose $s\ge t$. Then $\Phi(x-y,t-s)=0$.
\end{proof}

\begin{lem}\label{lem2.10}
 Suppose $\alpha>0$ and $T$ are constants.  Then for $s<t\leq T$ and $|x|\leq\sqrt{T-t}$ we have
 $$\int_{|y|<\sqrt{T-s}}\Phi(x-y,t-s)^{\alpha/n}dy\geq\frac{C}{(t-s)^{(\alpha-n)/2}}$$
 where $C=C(n,\alpha)$ is a positive constant. 
\end{lem}

\begin{proof}
 Making the change of variables $z=\frac{x-y}{\sqrt{t-s}}$ 
and letting $e_1 =(1,0,...,0)$ we get
 \begin{align}
  \notag \int_{|y|<\sqrt{T-s}}\Phi(x-y,t-s)^{\alpha/n}dy&=\frac{1}{(4\pi)^{\alpha/2}}\frac{1}{(t-s)^{\alpha/2}}\int_{|y|<\sqrt{T-s}}e^{-\frac{\alpha|x-y|^2}{4n(t-s)}}dy\notag\\
  &=\frac{1}{(4\pi)^{\alpha/2}(t-s)^{(\alpha-n)/2}}\int_{|z-\frac{x}{\sqrt{t-s}}|<\frac{\sqrt{T-s}}{\sqrt{t-s}}}e^{-\frac{\alpha}{4n}|z|^2}dz \label{2.26}\\
  &\geq\frac{1}{(4\pi)^{\alpha/2}}\frac{1}{(t-s)^{(\alpha-n)/2}}\int_{|z-\frac{\sqrt{T-s}}{\sqrt{t-s}}e_1|<\frac{\sqrt{T-s}}{\sqrt{t-s}}}e^{-\frac{\alpha}{4n}|z|^2}dz
\label{2.27}\\
  &\geq\frac{1}{(4\pi)^{\alpha/2}}\frac{1}{(t-s)^{(\alpha-n)/2}}\int_{|z-e_1
    |<1}e^{-\frac{\alpha}{4n}|z|^2}dz, \label{2.28}
 \end{align}
 where the last two inequalities need some explanation.  Since $|x|\leq\sqrt{T-t}<\sqrt{T-s}$, the center of the ball of integration in \eqref{2.26} is closer to the origin than the center of the ball of integration in \eqref{2.27}.  Thus, since the integrand 
 is a decreasing function of $|z|$, we obtain \eqref{2.27}.  Since $\sqrt{T-s}\geq\sqrt{t-s}$, the ball of integration in \eqref{2.27} contains the ball of integration in \eqref{2.28} and hence \eqref{2.28} holds.
\end{proof}

\section{The case $0<\lambda<\frac{n+2-\alpha}{n}$}\label{sec3}

In this section we prove Theorems \ref{thm1.1}, \ref{thm1.3}, and
\ref{thm1.4} when $0<\lambda<(n+2-\alpha)/n$.  For these values of
$\lambda$, Remark \ref{rem2.1} and the following theorem imply Theorems \ref{thm1.1} and
\ref{thm1.4}.

\begin{thm}\label{thm3.1}
 Suppose $u$ is a nonnegative solution of \eqref{2.10},\eqref{2.11} for some constants $\alpha\in(0,n+2)$,
 \begin{equation}\label{3.1}
  0<\lambda<\frac{n+2-\alpha}{n}\quad\text{and}\quad 0\leq\sigma\leq\frac{n+2}{n}.
 \end{equation}
 Then 
 \begin{equation}\label{3.2}
  \max_{|x|\leq1}u(x,t)=O(t^{-n/2})\quad\text{as }t\to0^+ .
 \end{equation}
\end{thm}

\begin{proof}
  Let $v=u+1$.  Then by Lemma \ref{lem2.3} we have \eqref{2.12}--\eqref{2.15}
  hold.  To prove \eqref{3.2}, it clearly suffices to prove
 \begin{equation}\label{3.3}
  \max_{|x|\leq1}v(x,t)=O(t^{-n/2})\quad\text{as }t\to0^+ .
 \end{equation}
 Choose $\varepsilon\in(0,1)$ such that
 \begin{equation}\label{3.4}
  \lambda<\frac{n+2-\alpha}{n+\varepsilon}.
 \end{equation}
 By \eqref{2.14},
 $$v^\lambda \in L^{\frac{n+2}{(n+\varepsilon)\lambda}}(\mathcal{P}_8 (0,8)).$$
 Thus, since \eqref{3.4} implies
 $$\frac{\lambda(n+\varepsilon)}{n+2}<\frac{n+2-\alpha}{n+2}$$
 we have by Theorem \ref{thmB.2} (with $\alpha$ replaced with
 $n+2-\alpha$) that
 $$\Phi^{\alpha/n}*v^\lambda \in L^\infty (\mathcal{P}_8 (0,8))$$
 where the convolution operation is in $\mathcal{P}_8 (0,8)$.  Hence by
 \eqref{2.12} and \eqref{2.13}, $v$ is a $C^{2,1}$ positive solution
 of
 $$0\leq Hv\leq Cv^\sigma \quad\text{in }B_2 (0)\times(0,8).$$
 Thus by \eqref{3.1}$_2$ and \cite[Theorem 1.1]{T2011}, $v$ satisfies
 \eqref{3.3}.
\end{proof}

The following theorem implies Theorem \ref{thm1.3} when $0<\lambda<\frac{n+2}{n}$.

\begin{thm}\label{thm3.2}
 Suppose $\alpha,\lambda$, and $\sigma$ are constants satisfying 
 \begin{equation}\label{3.5}
  \alpha\in(0,n+2),\quad 0<\lambda<\frac{n+2}{n},\quad\text{and}\quad \sigma>\frac{2(n+2)-\alpha}{n}-\lambda.
 \end{equation}
 Let $\varphi:(0,1)\to(0,\infty)$ be a continuous function satisfying
 $$\lim_{t\to0^+}\varphi(t)=\infty.$$
 Then there exists a positive function
 \begin{equation}\label{3.6}
  u\in C^\infty (\mathbb{R}^n \times(0,1))\cap L^\lambda (\mathbb{R}^n \times(0,1))
 \end{equation}
 satisfying
 \begin{equation}\label{3.7}
  0\leq Hu\leq(\Phi^{\alpha/n}*u^\lambda )u^\sigma \quad\text{in }\mathbb{R}^n \times(0,1),
 \end{equation}
 where $*$ is the convolution operation in $\mathbb{R}^n \times(0,1)$, such that
 \begin{equation}\label{3.8}
  u(0,t)\neq O(\varphi(t))\quad\text{as }t\to0^+ .
 \end{equation}
\end{thm}

 \begin{proof}
   By scaling $u$ and noting by \eqref{3.5} that
   $\sigma+\lambda\not= 1$ we see that it suffices to prove Theorem
   \ref{thm3.2} with \eqref{3.7} replaced with the weaker statement
   that there exists a positive constant
   $C=C(n,\lambda,\sigma,\alpha)$ such that $u$ satisfies
  \begin{equation}\label{3.9}
   0\leq Hu\leq C(\Phi^{\alpha/n}*u^\lambda )u^\sigma \quad\text{in }\mathbb{R}^n \times(0,1)
  \end{equation}
 where $*$ is the convolution operation in $\mathbb{R}^n \times(0,1)$.
 
 By \eqref{3.5} there exists $\varepsilon=\varepsilon(n,\lambda,\sigma,\alpha)\in(0,1)$ such that
 \begin{equation}\label{3.10}
  \sigma>\frac{2(n+2)-\alpha}{n-\varepsilon}-\lambda.
 \end{equation}
 Let
 \begin{equation}\label{3.11}
  p=\frac{n-\varepsilon}{2}
 \end{equation}
 and let $\{T_j \}\subset(0,1)$ be a sequence that $T_j \to0$ as $j\to\infty$.  Define $w_j :(-\infty,T_j )\to(0,\infty)$ by
 \begin{equation}\label{3.12}
  w_j (t)=(T_j -t)^{-p}
 \end{equation}
 and define $t_j \in(0,T_j )$ by
 \begin{equation}\label{3.13}
  w_j (t_j )=t^{-n/2}_{j}.
 \end{equation}
 Then
 \begin{equation}\label{3.14}
  \frac{T_j -t_j}{t_j}=\frac{w_j (t_j )^{-1/p}}{t_j}=t^{n/(2p)-1}_{j}\to0\quad\text{as }j\to\infty
 \end{equation}
 by \eqref{3.11}.
 
 Choose $a_j \in((t_j +T_j )/2,T_j )$ such that $w_j (a_j )>j\varphi(a_j )$.  Then
 \begin{equation}\label{3.15}
  \frac{w_j (a_j )}{\varphi(a_j )}\to\infty\quad\text{as }j\to\infty.
 \end{equation}
 Let $h_j (s)=\sqrt{a_j -s}$ and $H_j (s)=\sqrt{a_j +\varepsilon_j -s}$ where $\varepsilon_j >0$ satisfies
 \begin{equation}\label{3.16}
  a_j +2\varepsilon_j <T_j,
\quad t_j -\varepsilon_j >t_j /2,
\quad \varepsilon_j <T^{2}_{j},
\quad\text{and}\quad w_j (t_j -\varepsilon_j )>\frac{w_j (t_j )}{2}.
 \end{equation}
 Define 
 \begin{align*}
  &\omega_j =\{(y,s)\in\mathbb{R}^n \times\mathbb{R}:|y|<h_j (s)\quad\text{and}\quad t_j <s<a_j \}\\
  &\Omega_j =\{(y,s)\in\mathbb{R}^n \times\mathbb{R}:|y|<H_j(s)\quad\text{and}\quad t_j -\varepsilon_j <s<a_j +\varepsilon_j \}.
 \end{align*}
 By taking a subsequence we can assume the sets $\Omega_j$ are pairwise disjoint.
 
 Let $\chi_j :\mathbb{R}^n \times\mathbb{R}\to[0,1]$ be a $C^\infty$ function such that $\chi_j \equiv1$ in $w_j$ and $\chi_j \equiv0$ in $\mathbb{R}^n \times\mathbb{R}\backslash\Omega_j$.
 Define $f_j ,\,u_j :\mathbb{R}^n \times\mathbb{R}\to[0,\infty)$ by
 \begin{equation}\label{3.17}
  f_j (y,s)=\chi_j (y,s)w^{\prime}_{j}(s)
 \end{equation}
 and
 \begin{equation}\label{3.18}
  u_j (x,t)=\iint_{\mathbb{R}^n \times\mathbb{R}}\Phi(x-y,t-s)f_j (y,s)\,dy\,ds.
 \end{equation}
 Then $f_j$ and $u_j$ are $C^\infty$ and
 \begin{equation}\label{3.19}
  Hu_j =f_j \quad\text{in }\mathbb{R}^n \times\mathbb{R}.
 \end{equation}
 By Theorem \ref{thmB.2} with $p=n+2$ and $q=\infty$ we see that
 \begin{align}\label{3.20}
  \notag\|\iint_{\Omega_j \backslash \omega_j}&\Phi(x-y,t-s)w^{\prime}_{j}(s)\,dy\,ds\|_{L^\infty (\mathbb{R}^n \times(0,1))}\\
  \notag&=\|\iint_{\mathbb{R}^n \times(0,1)}\Phi(x-y,t-s)\chi_{\Omega_j \backslash \omega_j}(y,s)w^{\prime}_{j}(s)\,dy\,ds\|_{L^\infty (\mathbb{R}^n \times(0,1))}\\
  \notag&\leq C_n \| w^{\prime}_{j}(s)\|_{L^{n+2}(\Omega_j \backslash \omega_j )}\\
  &\leq w_j (t_j )
 \end{align}
 provided we decrease $\varepsilon_j$ if necessary because $|\Omega_j \backslash \omega_j |\to0$ as $\varepsilon_j \to0$.
 
 Also, for $(x,t)\in\Omega_j$ we have $|x|<\sqrt{T_j -t_j}$ by \eqref{3.16}$_1$, and thus using \eqref{3.16}$_2$ we obtain
 $$\sup_{(x,t)\in\Omega_j}\frac{|x|^2}{t}\leq\frac{T_j -t_j}{t_j -\varepsilon_j}\leq\frac{2(T_j -t_j )}{t_j}\to0\quad\text{as }j\to\infty$$
 by \eqref{3.14}.  Hence by \eqref{3.14}, \eqref{3.16}$_2$, and
 \eqref{3.13} there exists a positive number $M$, independent of $j$,
 such that
 \begin{equation}\label{3.21}
  M\Phi(x,t)\geq 2/t^{n/2}_{j}=2w_j (t_j)\quad\text{for }(x,t)\in\Omega_j .
 \end{equation}

 In order to obtain a lower bound for $u_j$ in $\Omega_j$, note first
 that for $s<t\leq a_j +\varepsilon_j$ and $|x|\leq H_j (t)$ we have by
 Lemma \ref{lem2.10} that
 \begin{equation}\label{3.22}
  \int_{|y|<H_j (s)}\Phi(x-y,t-s)\, dy\geq\beta
 \end{equation}
 for some constant
 \begin{equation}\label{3.23}
  \beta=\beta(n)\in(0,1).
 \end{equation}
 Next using \eqref{3.22} and \eqref{3.23}, we find for $(x,t)\in\Omega_j$ that
 \begin{align*}
  \iint_{\Omega_j}\Phi(x-y,t-s)w^{\prime}_{j}(s)\,dy\,ds&=\int^{t}_{t_j -\varepsilon_j}w^{\prime}_{j}(s)\left(\int_{|y|<H_j (s)}\Phi(x-y,t-s)dy\right)ds\\
  &\geq\beta(w_j (t)-w_j (t_j -\varepsilon_j ))\\
  &\geq\beta w_j (t)-w_j (t_j ).
 \end{align*}
 It therefore follows from \eqref{3.17}, \eqref{3.18}, and \eqref{3.20} that for $(x,t)\in\Omega_j$ we have
 \begin{align}\label{3.24}
  \notag u_j (x,t)&\geq\iint_{\omega_j}\Phi(x-y,t-s)w^{\prime}_{j}(s)\,dy\,ds\\
  \notag &=\iint_{\Omega_j}\Phi(x-y,t-s)w^{\prime}_{j}(s)\,dy\,ds
-\iint_{\Omega_j \backslash \omega_j}\Phi(x-y,t-s)w^{\prime}_{j}(s)\,dy\,ds\\
  &\geq\beta w_j (t)-2w_j (t_j ).
 \end{align}
 Also by \eqref{3.17}, \eqref{3.12}, and \eqref{3.16} we obtain
 \begin{align}\label{3.25}
  \notag \iint_{\mathbb{R}^n \times\mathbb{R}}f_j (y,s)\,dy\,ds&\leq\iint_{\Omega_j}w^{\prime}_{j}(s)\,dy\,ds\\
  \notag &\leq p\int^{T_j}_{0}(T_j -s)^{-(p+1)}\left(\int_{|y|<\sqrt{T_j -s}}dy\right)ds\\
  \notag &=p|B_1(0)|\int^{T_j}_{0}(T_j -s)^{n/2-p-1}ds\\
  &=p|B_1 (0)|\int^{T_j}_{0}\tau^{n/2-p-1}d\tau\to0\quad\text{as }j\to\infty
 \end{align}
 by \eqref{3.11}.  Hence for $1\leq\lambda<(n+2)/n$ it follows from \eqref{3.18} and Theorem \ref{thmB.2} that
 \begin{equation}\label{3.26}
  \| u_j \|_{L^\lambda (\mathbb{R}^n \times(0,1))}\to0\quad\text{as }j\to\infty.
 \end{equation}

 We next prove \eqref{3.26} when
 \begin{equation}\label{3.27}
  0<\lambda<1.
 \end{equation}
 (Theorem \ref{thmB.2} cannot be directly used in this case.)  Choose $z_0 >1$ such that the expression $z^{n/2}e^{-z/4}$ is decreasing on the interval $z_0 \leq z<\infty$.  Let 
 $r_0 =\sqrt{z_0}+1$.  Then $r_0 >2$ and by \eqref{3.17} and \eqref{3.18} we have
 \begin{align}\label{3.28}
  \notag\iint_{\mathbb{R}^n \times(0,1)}u_j (x,t)^\lambda \,dx\,dt&=\iint_{\mathbb{R}^n \times(0,1)}\left(\iint_{\mathbb{R}^n \times(0,1)}\Phi(x-y,t-s)f_j (y,s)\,dy\,ds\right)^\lambda \,dx\,dt\\
  &=I_j +J_j 
 \end{align}
 where
 $$I_j :=\iint_{B_{r_0}(0)\times(0,1)}\left(\iint_{\mathbb{R}^n \times(0,1)}\Phi(x-y,t-s)f_j (y,s)\,dy\,ds\right)^\lambda \,dx\,dt$$
 and
 $$J_j :=\iint_{(\mathbb{R}^n \backslash B_{r_0}(0))\times(0,1)}\left(\iint_{\Omega_j}\Phi(x-y,t-s)f_j (y,s)\,dy\,ds\right)^\lambda \,dx\,dt.$$
 By \eqref{3.27} and H\"older's inequality
 \begin{align}\label{3.29}
  \notag I_j &\leq\left(\iint_{B_{r_0}(0)\times(0,1)}\,dx\,dt\right)^{1-\lambda}\left(\iint_{B_{r_0}(0)\times(0,1)}\left(\iint_{\mathbb{R}^n \times(0,1)}\Phi(x-y,t-s)f_j (y,s)\,dy\,ds\right)\,dx\,dt\right)^\lambda \\
  &\to0\quad\text{as }j\to\infty
 \end{align}
 by \eqref{3.25} and Theorem \ref{thmB.2} with $p=q=1$.  Also
 \begin{equation}\label{3.30}
  J_j \leq\iint_{(\mathbb{R}^n \backslash B_{r_0}(0))\times(0,1)}A_j (x,t)^\lambda \,dx\,dt\| f_j \|^{\lambda}_{L^1 (\Omega_j )}
 \end{equation}
 where
 $$A_j (x,t):=\max_{(y,s)\in\Omega_j ,\,s<t}\Phi(x-y,t-s).$$
 For $s<t,\,(y,s)\in\Omega_j$ and $(x,t)\in(\mathbb{R}^n \backslash B_{r_0}(0))\times(0,1)$ we have $0<s<t<1$ and
 $$|x-y|>|x|-|y|>|x|-1.$$
 Thus
 \begin{align}\label{3.31}
  \notag (4\pi)^{n/2}\Phi(x-y,t-s)&\leq\frac{1}{(t-s)^{n/2}}e^{-\frac{(|x|-1)^2}{4(t-s)}}\\
  &=\frac{1}{(|x|-1)^n}\left(\frac{(|x|-1)^2}{t-s}\right)^{n/2}e^{-\frac{(|x|-1)^2}{4(t-s)}}.
 \end{align}
 Since $|x|\geq r_0$ and $0<s<t<1$ we have
 $$\frac{(|x|-1)^2}{t-s}>(|x|-1)^2 \geq z_0$$
 and thus by the definition of $z_0$ we obtain from \eqref{3.31} that
 \begin{align*}
  (4\pi)^{n/2}\Phi(x-y,t-s)&\leq\frac{1}{(|x|-1)^n}((|x|-1)^2 )^{n/2}e^{-(|x|-1)^2 /4}\\
  &=e^{-(|x|-1)^2 /4}.
 \end{align*}
 Hence
 $$A_j (x,t)^\lambda \leq e^{-\lambda(|x|-1)^2 /4}\quad\text{for }(x,t)\in(\mathbb{R}^n \backslash B_{r_0}(0))\times(0,1).$$
 It therefore follows from \eqref{3.30} and \eqref{3.25} that $J_j \to0$ as $j\to\infty$ which together with \eqref{3.29} and \eqref{3.28} yields \eqref{3.26} when $\lambda$ satisfies
 \eqref{3.27}.
 
 By \eqref{3.25} we find that
 $$\iint_{\mathbb{R}^n \times\mathbb{R}}\sum^{\infty}_{j=1}f_j (y,s)\,dy\,ds<\infty$$
 provided we take a subsequence if necessary.  Hence, since the $C^\infty$
 functions 
$f_j$ have disjoint supports, we see that the function 
 $u:(\mathbb{R}^n \times\mathbb{R})\backslash \{(0,0)\}\to[0,\infty)$ defined by
 \begin{equation}\label{3.32}
  u(x,t)=(M+1)\Phi(x,t)+\sum^{\infty}_{j=1}u_j (x,t)
 \end{equation}
 is $C^\infty$ and by \eqref{3.18} we have
 \begin{equation}\label{3.33}
  Hu=\sum^{\infty}_{j=1}f_j \quad\text{in }(\mathbb{R}^n \times\mathbb{R})\backslash \{(0,0)\}
 \end{equation}
 $$u\equiv0\quad\text{in }\mathbb{R}^n \times(-\infty,0).$$
 From \eqref{3.26} we have
 $$u\in L^\lambda (\mathbb{R}^n \times(0,1))$$
 provided we take a subsequence of $u_j$ if necessary.  Thus \eqref{3.6} holds.
 
 We now prove \eqref{3.9}.  By \eqref{3.33} and \eqref{3.17} we have $Hu\equiv0$ in $(\mathbb{R}^n \times(0,1))\backslash \bigcup^{\infty}_{j=1}\Omega_j .$  Hence to prove \eqref{3.9}, it
 suffice to prove there exists a positve constant 
$C=C(n,\lambda,\sigma,\alpha)$ such that
 \begin{equation}\label{3.34}
  0\leq Hu\leq C(\Phi^{\alpha/n}*u^\lambda )u^\sigma \quad\text{in }\Omega_j
 \end{equation}
 for $j=1,2,...$.
 
 By \eqref{3.32}, \eqref{3.24}, and \eqref{3.21} we have for $(x,t)\in\Omega_j$ that
 \begin{align}\label{3.35}
  \notag u(x,t)&\geq(M+1)\Phi(x,t)+\beta w_j (t)-2w_j (t_j )\\
  &\geq\Phi(x,t)+\beta w_j (t).
 \end{align}
 Thus for $(x,t)\in\Omega_j$ we see by \eqref{3.33}, \eqref{3.17}, and \eqref{3.12} that
 \begin{align*}
  Hu(x,t)&=f_j (x,t)\leq w^{\prime}_{j}(t)=pw_j (t)^{1+1/p}\\
  &=pw_j (t)^{1+1/p-\sigma}w_j (t)^\sigma 
\leq\frac{p}{\beta^\sigma}w_j (t)^{1+1/p-\sigma}u(x,t)^\sigma.
 \end{align*}
 Hence to prove \eqref{3.34} it suffices to show
 \begin{equation}\label{3.36}
  w_j (t)^{1+1/p-\sigma}<C\iint_{\mathbb{R}^n  \times(0,1)}\Phi(x-y,t-s)^{\alpha/n}u(y,s)^\lambda \,dy\,ds\quad\text{for }(x,t)\in\Omega_j .
 \end{equation}
 Our proof of \eqref{3.36} consists of two cases.\\
 \noindent \textbf{Case I.} Suppose
 \begin{equation}\label{3.37}
  (x,t)\in\Omega_j \quad\text{and}\quad t\leq\frac{T_j +t_j}{2}.
 \end{equation}
 Then using \eqref{3.16}, \eqref{3.11}, and the fact that $w_j$ is an increasing function we have
 $$\frac{1}{2}\leq\frac{w_j (t)}{2w_j (t_j -\varepsilon_j )}\leq\frac{w_j (t)}{w_j (t_j )}\leq\left(\frac{T_j -\frac{T_j +t_j}{2}}{T_j -t_j}\right)^{-p}=2^p <2^{n/2}.$$
 Also by \eqref{3.13} and \eqref{3.14}
 $$\frac{w_j (t_j )}{T^{-n/2}_{j}}=\left(\frac{T_j}{t_j}\right)^{n/2}\in(1,2)$$
 provided we take a subsequence if necessary.  Thus \eqref{3.37} implies
 \begin{equation}\label{3.38}
  \frac{1}{2}<\frac{w_j (t)}{T^{-n/2}_{j}}<2^{(n+2)/2}.
 \end{equation}
 Next, making the change of variables
 $$x=\sqrt{T_j}\xi,\quad t=T_j \tau,\quad\text{and}\quad
 y=\sqrt{T_j}\eta,\quad s=T_j \zeta,$$
 we get
 \begin{align}\label{3.39}
  \notag \iint_{\mathbb{R}^n \times(0,1)}&\Phi(x-y,t-s)^{\alpha/n}\Phi(y,s)^\lambda \,dy\,ds\\
  \notag &=\iint_{\mathbb{R}^n \times(0,\tau)}\frac{1}{T^{\frac{n}{2}\frac{\alpha}{n}}_{j}}\Phi(\xi-\eta,\tau-\zeta)^{\alpha/n}\frac{1}{T^{n\lambda/2}_{j}}\Phi(\eta,\zeta)^\lambda \sqrt{T_j}^{n+2}\,d\eta\,d\zeta\\
  &\geq\frac{G(\xi,\tau)}{\sqrt{T_j}^{\alpha+n\lambda-(n+2)}}
 \end{align}
 where
 $$G(\xi,\tau):=\iint_{B_1(0)\times(1/2,\tau)}\Phi(\xi-\eta,\tau-\zeta)^{\alpha/n}\Phi(\eta,\zeta)^\lambda \,d\eta\,d\zeta.$$
 Since by \eqref{3.37}$_1$, \eqref{3.16}$_1$, \eqref{3.14}, and
 \eqref{3.16}$_3$,
 $$1>\tau=\frac{t}{T_j}\geq\frac{t_j -\varepsilon_j}{T_j}\to1\quad\text{as }j\to\infty$$
 we have by \eqref{3.37}$_1$ that
 $$|\xi|=\frac{|x|}{\sqrt{T_j}}<\frac{\sqrt{T_j -t}}{\sqrt{T_j}}=\sqrt{1-\frac{t}{T_j}}\to0\quad\text{as }j\to\infty.$$
 Thus, since $G$ is clearly continuous at $(\xi,\tau)=(0,1)$ and $G(0,1)>0$ we have by \eqref{3.39} that
 \begin{equation}\label{3.40}
  \iint_{\mathbb{R}^n \times(0,1)}\Phi(x-y,t-s)^{\alpha/n}\Phi(y,s)^\lambda \,dy\,ds\geq\frac{C}{\sqrt{T_j}^{\alpha+n\lambda-(n+2)}}\quad\text{for }(x,t)\in\Omega_j,
 \end{equation}
 where $C:=G(0,1)/2>0$, provided we take a subsequence if necessary.
 
 Since by \eqref{3.10} and \eqref{3.11},
 \begin{align*}
  \sigma&>\frac{n+2}{n-\varepsilon}+\frac{n+2-\alpha}{n-\varepsilon}-\lambda\\
  &>\frac{n+2-\varepsilon}{n-\varepsilon}+\frac{n+2-\alpha}{n}-\lambda\\
  &=\frac{1}{p}+1+\frac{n+2-\alpha}{n}-\lambda
 \end{align*}
 we have
 \begin{align*}
  \left(\frac{1}{p}+1-\sigma\right)\frac{n}{2}&<\frac{n}{2}\left(\lambda-\frac{n+2-\alpha}{n}\right)\\
  &=\frac{\alpha+n\lambda-(n+2)}{2}.
 \end{align*}
 Thus \eqref{3.36} follows from \eqref{3.35}, \eqref{3.38}, and \eqref{3.40}.\\
 \textbf{Case II.} Suppose
 \begin{equation}\label{3.41}
  (x,t)\in\Omega_j \quad\text{and}\quad t\geq\frac{T_j +t_j}{2}.
 \end{equation}
 Then for $s<t$ we have by Lemma \ref{lem2.10} with
 $T=a_j+\varepsilon_j$ that
 $$\int_{|y|<H_j (s)}\Phi(x-y,t-s)^{\alpha/n}dy\geq\frac{C}{(t-s)^{(\alpha-n)/2}}$$ 
 for some positive constant $C=C(n,\alpha)$.  Thus for $(x,t)$ satisfying \eqref{3.41} we get
 \begin{align}\label{3.42}
  \notag \iint_{\Omega_j}\Phi&(x-y,t-s)^{\alpha/n}w_j (s)^\lambda \,dy\,ds\geq\int^{t}_{t_j}w_j (s)^\lambda \left(\int_{|y|<H_j (s)}\Phi(x-y,t-s)^{\alpha/n}dy\right)ds\\
  \notag &\geq C\int^{t}_{t_j}\frac{ds}{(t-s)^a (T_j -s)^b}\quad
           \text{where }a=(\alpha-n)/2 \text{ and } b=\lambda p\\
  \notag &=\frac{C}{(T_j -t)^{a+b-1}}\int^{\frac{T_j -t_j}{T_j -t}}_{1}\frac{dz}{(z-1)^a z^b}\text{ under the change of variables }T_j -s=(T_j -t)z\\
  &\geq\frac{C}{(T_j -t)^{a+b-1}}\int^{2}_{1}\frac{dz}{(z-1)^a z^b}
=\frac{C}{(T_j -t)^{(\alpha-n)/2+\lambda p-1}}.
 \end{align}

 Since by \eqref{3.10} and \eqref{3.11}
 \begin{align*}
  p(\sigma+\lambda-1)&>\frac{n-\varepsilon}{2}\left(\frac{2(n+2)-\alpha}{n-\varepsilon}-1\right)\\
  &=n+2-\frac{\alpha}{2}-\frac{n}{2}+\frac{\varepsilon}{2}\\
  &>\frac{n-\alpha}{2}+2,
 \end{align*}
 we have
 $$p(1+\frac{1}{p}-\sigma)<\frac{\alpha-n}{2}+\lambda p-1.$$
 Thus \eqref{3.36} follows from \eqref{3.12}, \eqref{3.35}, and \eqref{3.42}.  This completes the proof of \eqref{3.36} in all cases.  Hence \eqref{3.34} and \eqref{3.9} hold.
 
 Finally \eqref{3.8} follows from \eqref{3.15} and \eqref{3.35} 
with $(x,t)=(0,a_j)$.
\end{proof}

\section{The case
  $\frac{n+2-\alpha}{n}\leq\lambda<\frac{n+2}{n}$}\label{sec4}

In this section we prove Theorem \ref{thm1.1} when
\begin{equation}\label{4.1}
 \frac{n+2-\alpha}{n}\leq\lambda<\frac{n+2}{n}.
\end{equation}
(For these values of $\lambda$, Theorem \ref{thm1.3} follows from
Theorem \ref{thm3.2} in the last section and Theorems \ref{thm1.2} and
\ref{thm1.4} are vacuously true.)

For $\lambda$ satisfying \eqref{4.1}, Remark \ref{rem2.1} and the following theorem imply Theorem \ref{thm1.1}.

\begin{thm}\label{thm4.1}
 Suppose $u$ is a nonnegative solution of \eqref{2.10},\eqref{2.11} for some constants $\alpha\in(0,n+2)$,
 \begin{equation}\label{4.2}
  \frac{n+2-\alpha}{n}\leq\lambda<\frac{n+2}{n}\quad\text{and}\quad 0\leq\sigma<\frac{2(n+2)-\alpha}{n}-\lambda.
 \end{equation}
 Then
 \begin{equation}\label{4.3}
  \max_{|x|\leq1}u(x,t)=O(t^{-n/2})\quad\text{as }t\to0^+ .
 \end{equation}
\end{thm}

\begin{proof}
 Let $v=u+1$.  Then by Lemma \ref{lem2.3} we have that \eqref{2.12}--\eqref{2.15} hold.  To prove \eqref{4.3}, it clearly suffices to prove
 \begin{equation}\label{4.4}
  \max_{|x|\leq1}v(x,t)=O(t^{-n/2})\quad\text{as }t\to0^+ .
 \end{equation}

 Since increasing $\lambda$ or $\sigma$ increases the right side of the second inequality in \eqref{2.13}$_1$, we can assume instead of \eqref{4.2} that
 \begin{equation}\label{4.5}
  \frac{n+2-\alpha}{n}<\lambda<\frac{n+2}{n},\quad \sigma>0,\quad\text{and}\quad 1<\lambda+\sigma<\frac{2(n+2)-\alpha}{n}.
 \end{equation}
 Since the increased value of $\lambda$ is less than $\frac{n+2}{n}$, it follows from \eqref{2.14} that \eqref{2.12} still holds.
 
 By \eqref{4.5} there exists $\varepsilon=\varepsilon(n,\lambda,\sigma,\alpha)\in(0,1)$ such that
 \begin{equation}\label{4.6}
  \left(\frac{n+4-\alpha}{n+4-\alpha-\varepsilon}\right)\frac{n+2-\alpha}{n}<\lambda<\frac{n+2}{n+\varepsilon}\quad\text{and}\quad \lambda+\sigma<\frac{2(n+2)-\alpha}{n+\varepsilon}
 \end{equation}
 which implies
 \begin{equation}\label{4.7}
  \sigma<\frac{2(n+2)-\alpha}{n+\varepsilon}-\lambda<\frac{2(n+2)-\alpha}{n+\varepsilon}-\frac{n+2-\alpha}{n}<\frac{n+2}{n+\varepsilon}.
 \end{equation}

 Suppose for contradition that \eqref{4.4} is false.  Then there is a
 sequence $\{(x_j ,t_j )\}\subset\overline{B_1(0)}\times(0,1)$ and
 $x_0 \in\overline{B_1 (0)}$ such that $(x_j ,t_j )\to(x_0 ,0)$
 as $j\to\infty$ and
 \begin{equation}\label{4.8}
  \lim_{j\to\infty}t^{n/2}_{j}v(x_j ,t_j )=\infty.
 \end{equation}
 By Lemma \ref{lem2.8} we have for $(x,t)\in\mathcal{P}_{t_j /4}(x_j ,t_j )$ that
 $$\iint_{\mathcal{P}_8 (0,8)\backslash\mathcal{P}_{t_j /2}(x_j ,t_j )}\Phi(x-y,t-s)Hv(y,s)\,dy\,ds\leq\frac{C(n)}{t^{n/2}_{j}}\iint_{\mathcal{P}_8 (0,8)}Hv(y,s)\,dy\,ds.$$
 It therefore follows from \eqref{2.14} and Remark \ref{rem2.2} that
 \begin{equation}\label{4.9}
  v(x,t)\leq C\left[\left(\frac{1}{\sqrt{t}}\right)^n +\iint_{\mathcal{P}_{t_j /2}(x_j ,t_j ))}\Phi(x-y,t-s)Hv(y,s)\,dy\,ds\right]\quad\text{for }(x,t)\in\mathcal{P}_{t_j /4}(x_j ,t_j ).
 \end{equation}
 Substituting $x=x_j$ and $t=t_j$ in \eqref{4.9} and using \eqref{4.8} we find that
 \begin{equation}\label{4.10}
  t^{n/2}_{j}\iint_{\mathcal{P}_{t_j /2}(x_j ,t_j )}\Phi(x_j -y,t_j -s)Hv(y,s)\,dy\,ds\to\infty\quad\text{as }j\to\infty.
 \end{equation}
 Also, by \eqref{2.14} we have 
 \begin{equation}\label{4.11}
  \iint_{\mathcal{P}_{t_j /2}(x_j ,t_j )}Hv(y,s)\,dy\,ds\to0\quad\text{as }j\to\infty.
 \end{equation}
 Defining
 \begin{equation}\label{4.12}
  f_j (\eta,\zeta)=r^{\frac{n+2}{2}}_{j}Hv(x_j +\sqrt{r_j}\eta,t_j
  +r_j \zeta)\quad \text{where }r_j =t_j /8
 \end{equation}
 and making the change of variables
 \begin{equation}\label{4.13}
  y=x_j +\sqrt{r_j}\eta,\quad s=t_j +r_j \zeta
 \end{equation}
 in \eqref{4.11} and \eqref{4.10} we get
 
 \begin{equation}\label{4.14}
  \iint_{\mathcal{P}_4 (0,0)}f_j (\eta,\zeta)\,d\eta\,d\zeta\to0\quad\text{as }j\to\infty
 \end{equation}
 and
 \begin{equation}\label{4.15}
  \iint_{\mathcal{P}_4 (0,0)}\Phi(-\eta,-\zeta)f_j (\eta,\zeta)\,d\eta\,d\zeta\to\infty\quad\text{as }j\to\infty.
 \end{equation}
 Let
 $$N(y,s)=\iint_{\mathcal{P}_8
   (0,8)}\Phi(y-\bar y,s-\bar s)Hv(\bar y,\bar s)\,
d\bar y\, d\bar s.$$
 By \eqref{2.14} and Theorem \ref{thmB.2} we find that $N\in L^{\frac{n+2}{n+\varepsilon}}(\mathcal{P}_8 (0,8))$ and thus $N^\lambda \in L^{\frac{n+2}{\lambda(n+\varepsilon)}}(\mathcal{P}_8 (0,8))$.  Hence by H\"older's inequality, \eqref{4.6}, and Lemma \ref{lem2.5} we have for 
 $R\in(0,1]$ and $(x,t)\in\mathcal{P}_{Rt_j /8}(x_j ,t_j )$ that
 \begin{align}\label{4.16}
  \notag &\iint_{\mathcal{P}_8 (0,8)\backslash\mathcal{P}_{Rt_j /4}(x_j ,t_j )}
\Phi(x-y,t-s)^{\alpha/n}N(y,s)^\lambda \,dy\,ds\\
  \notag &\leq\| N^\lambda \|_{L^{\frac{n+2}{\lambda(n+\varepsilon)}}(\mathcal{P}_8 (0,8))}\left(\iint_{\mathbb{R}^n \times\mathbb{R}\backslash\mathcal{P}_{Rt_j /4}(x_j ,t_j )}\Phi(x-y,t-s)^{\frac{\alpha q}{n}}\,dy\,ds\right)^{1/q}\text{ where }\frac{\lambda(n+\varepsilon)}{n+2}+\frac{1}{q}=1\\
  &\leq C\left(\frac{1}{\sqrt{Rt_j}^{\alpha q-(n+2)}}\right)^{1/q}\\
  \label{4.17} &=C\frac{1}{\sqrt{t_j}^{(n+\varepsilon)\lambda-(n+2-\alpha)}}
\end{align}
where $C>0$ depends on $R$ but not on $j$.

 Since by \eqref{2.14} and Lemma \ref{lem2.8} we have
 $$N(y,s)\leq
 C\left[\frac{1}{\sqrt{t_j}^n}+\iint_{(\bar y,\bar s)\in\mathcal{P}_{Rt_j /2}(x_j ,t_j
     )}\Phi(y-\bar y,s-\bar s)Hv(\bar y,\bar s)\,
   d\bar y\, d\bar s\right]\quad\text{for }(y,s)\in P_{Rt_j /4}(x_j ,t_j )$$
 it follows from Lemma \ref{lem2.6} that for $(x,t)\in\mathcal{P}_{Rt_j /4}(x_j ,t_j )$ we have

 \begin{align}
  \iint_{(y,s)\in\mathcal{P}_{Rt_j /4}(x_j,t_j)}&
\Phi(x-y,t-s)^{\alpha/n}N(y,s)^\lambda \,dy\,ds\notag\\
  &\leq C\left[\frac{1}{\sqrt{t_j}^{n\lambda-(n+2-\alpha)}}
+\iint_{(y,s)\in\mathcal{P}_{Rt_j /4}(x_j,t_j)}\Phi(x-y,t-s)^{\alpha/n}\right.\notag\\
  &\quad\left.\times\left(\iint_{(\bar y,\bar s)\in\mathcal{P}_{Rt_j /2}(x_j,t_j )}\Phi(y-\bar y,s-\bar s)Hv(\bar y,\bar s)\,d\bar y\,d\bar
    s\right)^\lambda \,dy\,ds\right].
 \label{4.17.1}
\end{align}
Also by Jensen's inequality, \eqref{4.5} and Lemma \ref{lem2.7} we
 have for $x\in\mathbb{R}^n$, $t>0$, and $\lambda\ge 1$ that
 \begin{align}
  \iint_{\mathbb{R}^n \times\mathbb{R}}&\Phi(x-y,t-s)^{\alpha/n}\left(\int_{|z|<\sqrt{8}}\Phi(y-z,s)\,d\mu(z)\right)^\lambda \,dy\,ds\notag\\
  &\leq C\int_{|z|<\sqrt{8}}\left(\iint_{\mathbb{R}^n \times (0,t)}\Phi(x-y,t-s)^{\alpha/n}\Phi(y-z,s)^{\lambda n/n}\,dy\,ds\right)d\mu(z)\notag\\
  &\leq\frac{C}{\sqrt{t}^{\alpha+\lambda n-(n+2)}}.\label{4.17.2}
 \end{align}

We claim that \eqref{4.17.2} also holds for $0<\lambda<1$. To see
this, let $x\in\R^n$ and $t>0$ be fixed and define 
\[
f(y,s)=\Phi(x-y,t-s)^{\alpha/n} \quad\text{and}\quad 
g(y,s)=\int_{|z|<\sqrt{8}}\Phi(y-z,s)\,d\mu(z).
\]
Then by Lemma \ref{lem2.7} with $\beta=0$ and $\beta=n$ we have
\[
\|f\|_1:=\iint_{\R^n\times(0,t)}\Phi(x-y,t-s)^{\alpha/n}dy\,ds\le C\sqrt{t}^{n+2-\alpha}
\]
and
\[  
\iint_{\R^n\times(0,t)}fg\,dy\,ds
=\int_{|z|<\sqrt{8}}\iint_{\R^n\times(0,t)}\Phi(x-y,t-s)^{\alpha/n}\Phi(y-z,s)\,dy\,ds\,d\mu(z)\le
C\sqrt{t}^{2-\alpha},
\]
respectively, where $C$ depends on neither $x$ nor $t$. Thus by
Jensen's inequality we find for $(x,t)\in\R^n\times(0,\infty)$ and
$0<\lambda<1$ that
\begin{align*}
\iint_{\R^n\times\R}fg^\lambda dy\,ds
&=\iint_{\R^n\times(0,t)}\left(g\|f\|_1^{1/\lambda}\right)^\lambda
\frac{f}{\|f\|_1}\,dy\,ds\\
&\le \left(\iint_{\R^n\times(0,t)}g\|f\|_1^{1/\lambda}
\frac{f}{\|f\|_1}\,dy\,ds\right)^\lambda\\
&=\|f\|_1^{1-\lambda}\left(\iint_{\R^n\times(0,t)}fg\,dy\,ds\right)^\lambda
\le C\sqrt{t}^{n+2-\alpha-\lambda n}.
\end{align*}
That is \eqref{4.17.2} also holds for $0<\lambda<1$.

It therefore follows from \eqref{2.15}, \eqref{4.17}, \eqref{4.17.1}, and
Lemma \ref{lem2.7} that for $(x,t)\in\mathcal{P}_{Rt_j /8}(x_j ,t_j )$ we have
 \begin{align*}
  \int_{(y,s)\in\mathcal{P}_8 (0,8)}&\Phi(x-y,t-s)^{\alpha/n}v(y,s)^\lambda \,dy\,ds\\
  \begin{split}  
  \leq& C\left[\iint_{(y,s)\in\mathcal{P}_8(0,8)}\Phi(x-y,t-s)^{\alpha/n}
\left(\int_{|z|<\sqrt{8}}\Phi(y-z,s)\,d\mu(z)\right)^\lambda \,dy\,ds\right.\\
  &\left.+\iint_{(y,s)\in\mathcal{P}_8 (0,8)}\Phi(x-y,t-s)^{\alpha/n}(N(y,s)+1)^\lambda \,dy\,ds\right]\\
  \leq& C\left[\frac{1}{\sqrt{t_j}^{(n+\varepsilon)\lambda-(n+2-\alpha)}}+\iint_{(y,s)\in\mathcal{P}_{Rt_j /4}(x_j ,t_j )}\Phi(x-y,t-s)^{\alpha/n}\right.\\
  &\times\left.\left(\iint_{(\bar y,\bar s)\in\mathcal{P}_{Rt_j /2}(x_j ,t_j )}\Phi(y-\bar y,s-\bar s)Hv(\bar y,\bar s)\,d\bar y\,d\bar s\right)^\lambda \,dy\,ds\right]
  \end{split}
 \end{align*}
 where $C>0$ depends on $R$ but not on $j$.
 
 Also, similar to the way \eqref{4.9} was derived, we obtain
 $$v(x,t)\leq C\left[\frac{1}{\sqrt{t_j}^n}+\iint_{\mathcal{P}_{Rt_j /4}(x_j ,t_j )}\Phi(x-y,t-s)Hv(y,s)\,dy\,ds\right]\quad\text{for }(x,t)\in P_{Rt_j /8}(x_j ,t_j ).$$
 We see therefore from \eqref{2.13} that for $(x,t)\in\mathcal{P}_{Rt_j /8}(x_j ,t_j )$ and $R\in(0,1]$ we have 
 \begin{align*}
  \begin{split}
   Hv(x,t)&\leq C\left[\frac{1}{\sqrt{t_j}^{(n+\varepsilon)\lambda-(n+2-\alpha)}}+\iint_{(y,s)\in\mathcal{P}_{Rt_j /4}(x_j ,t_j )}\Phi(x-y,t-s)^{\alpha/n}\right.\\
   &\quad\times\left.\left(\iint_{(\bar y,\bar s)\in\mathcal{P}_{Rt_j /2}(x_j ,t_j )}\Phi(y-\bar y,s-\bar s)Hv(\bar y,\bar s)\,d\bar y\,d\bar s\right)^\lambda \,dy\,ds\right]\\
   &\quad\times\left[\frac{1}{\sqrt{t_j}^{n\sigma}}+\left(\iint_{\mathcal{P}_{Rt_j /2}(x_j ,t_j )}\Phi(x-y,t-s)Hv(y,s)\,dy\,ds\right)^\sigma \right].
  \end{split}
 \end{align*}
 Hence under the change of variables \eqref{4.13},
 $$x=x_j +\sqrt{r_j}\xi,\quad t=t_j +r_j \tau,$$
 and
 $$\bar y=x_j +\sqrt{r_j}\bar \eta,\quad \bar s=t_j +r_j \bar \zeta,$$
 we obtain from \eqref{4.12} and \eqref{4.6} that
 \begin{align}\label{4.18}
  \notag &f_j (\xi,\tau)=r^{\frac{n+2}{2}}_{j}Hv(x,t)\leq r^{\frac{(n+\varepsilon)(\lambda+\sigma)-(n+2-\alpha)}{2}}_{j}Hv(x,t)\\
  \notag &\leq C\left[1+\iint_{(\eta,\zeta)\in\mathcal{P}_{4R}(0,0)}\Phi(\xi-\eta,\tau-\zeta)^{\alpha/n}\left(\iint_{(\bar \eta,\bar \zeta)\in\mathcal{P}_{4r}(0,0)}\Phi(\eta-\bar \eta,\zeta-\bar \zeta)
f_j (\bar \eta,\bar \zeta)\,d\bar \eta\,d\bar \zeta\right)^\lambda d\eta\,d\zeta\right]\\
  &\quad\times\left[1+\left(\iint_{\mathcal{P}_{4R}(0,0)}\Phi(\xi-\eta,\tau-\zeta)f_j (\eta,\zeta)\,d\eta\,d\zeta\right)^\sigma \right]
 \end{align}
 for $(\xi,\tau)\in\mathcal{P}_R (0,0)$ and $R\in(0,1]$ where $C>0$
 depends on $R$ but not on $j$.

To complete the proof of Theorem \ref{thm4.1} we will need the following lemma.

\begin{lem}\label{lem4.1}
 Suppose the sequence
 \begin{equation}\label{4.19}
  \{f_j \}\text{ is bounded in }L^p (\mathcal{P}_{4R}(0,0))
 \end{equation}
 for some constants $p\in[1,\frac{n+2}{2}]$ and $R\in(0,1]$.  
Then there exists a positive constant $C_0 =C_0 (n,\lambda,\sigma,\alpha)$ such that the sequence 
\begin{equation}\label{4.20}
 \{f_j \}\text{ is bounded in }L^q (\mathcal{P}_R(0,0))
\end{equation}
for some $q\in(p,\infty)$ satisfying
\begin{equation}\label{4.21}
 \frac{1}{p}-\frac{1}{q}\geq C_0.
\end{equation}
\end{lem}

\begin{proof}
 For $R\in(0,1]$ we formally define operators $N_R$ and $I_R$ by
 \[
(N_R f)(\xi,\tau)=\iint_{\mathcal{P}_{4R}(0,0)}
\Phi(\xi-\eta,\tau-\zeta)f(\eta,\zeta)\,d\eta\,d\zeta
\]
and
\[
(I_R f)(\xi,\tau)=\iint_{\mathcal{P}_{4R}(0,0)}
\Phi(\xi-\eta,\tau-\zeta)^{\alpha/n}f(\eta,\zeta)\,d\eta\,d\zeta.
\]
 Define $p_2$ by
 \begin{equation}\label{4.22}
  \frac{1}{p}-\frac{1}{p_2}=\frac{2-\varepsilon}{n+2}
 \end{equation}
 where $\varepsilon$ is as in \eqref{4.6}.  Then $p_2 \in(p,\infty)$
 and thus by Theorem \ref{thmB.2} we have
 \begin{equation}\label{4.23}
  \|(N_R f_j )^\lambda \|_{p_2 /\lambda}=\| N_R f_j \|^{\lambda}_{p_2}\leq C\| f_j \|^{\lambda}_{p}
 \end{equation}
 and
 \begin{equation}\label{4.24}
  \|(N_R f_j )^\sigma \|_{p_2 /\sigma}=\| N_R f_j \|^{\sigma}_{p_2}\leq C\| f_j \|^{\sigma}_{p}
 \end{equation}
 where $\|\cdot\|_p :=\|\cdot\|_{L^p (\mathcal{P}_{4R}(0,0))}$. Since
 $$\frac{1}{p_2}=\frac{1}{p}-\frac{2-\varepsilon}{n+2}
\leq 1-\frac{2-\varepsilon}{n+2}=\frac{n+\varepsilon}{n+2}$$
 we see by \eqref{4.6} that
 \begin{equation}\label{4.25}
  \frac{p_2}{\lambda}>1.
 \end{equation}
 Now there are two cases to consider.
 \medskip

 \noindent \textbf{Case I.} Suppose
 \begin{equation}\label{4.26}
  \frac{p_2}{\lambda}<\frac{n+2}{n+2-\alpha}.
 \end{equation}
 Define $p_3$ and $q$ by
 \begin{equation}\label{4.27}
  \frac{\lambda}{p_2}-\frac{1}{p_3}=\frac{n+2-\alpha}{n+2}
 \end{equation}
 and
 \begin{equation}\label{4.28}
  \frac{1}{q}:=\frac{1}{p_3}+\frac{\sigma}{p_2}=\frac{\lambda+\sigma}{p_2}-\frac{n+2-\alpha}{n+2}.
 \end{equation}
 It follows from \eqref{4.25}--\eqref{4.28}, \eqref{4.22}, and \eqref{4.5} that
 \begin{equation}\label{4.29}
  1<\frac{p_2}{\lambda}<p_3 <\infty,\quad q>0,
 \end{equation}
 and
 \begin{align*}
  \frac{1}{p}-&\frac{1}{q}
=\frac{1}{p}-\left((\lambda+\sigma)\left(\frac{1}{p}-\frac{2-\varepsilon}{n+2}\right)-\frac{n+2-\alpha}{n+2}\right)\\
  &=\frac{(2-\varepsilon)(\lambda+\sigma)+(n+2-\alpha)}{n+2}-\frac{\lambda+\sigma-1}{p}\\
  &\geq\frac{(2-\varepsilon)(\lambda+\sigma)+(n+2-\alpha)-(n+2)(\lambda+\sigma-1)}{n+2}\\
  &=\frac{2(n+2)-\alpha-(n+\varepsilon)(\lambda+\sigma)}{n+2}.
 \end{align*}
 Thus \eqref{4.21} holds by \eqref{4.6}.
 
 By \eqref{4.27}, \eqref{4.29}, \eqref{4.23}, and Theorem \ref{B.1} we
 find that
 \begin{align*}
  \|(I_R ((N_R &f_j )^\lambda ))^q \|_{p_3/q}=\| I_R ((N_R f_j )^\lambda )\|^{q}_{p_3}\\
  &\leq C\|(N_R f_j )^\lambda \|^{q}_{p_2 /\lambda}\\
  &\leq C\| f_j \|^{\lambda q}_{p}.
 \end{align*}
 Also by \eqref{4.24} we get
 $$\|(N_R f_j )^{\sigma q}\|_{\frac{p_2}{\sigma q}}=\|(N_R f_j )^\sigma \|^{q}_{p_2 /\sigma}\leq C\| f_j \|^{\sigma q}_{p}.$$
 It therefore follows from \eqref{4.18}, \eqref{4.28}, H\"older's inequality, and \eqref{4.19} that \eqref{4.20} holds.
 \medskip

\noindent \textbf{Case II.} Suppose
\begin{equation}\label{4.30} 
\frac{p_2}{\lambda}\geq\frac{n+2}{n+2-\alpha}.
\end{equation} 
Then by Theorem \ref{thmB.2}, \eqref{4.19}, and \eqref{4.23} we
find that the sequence
 \begin{equation}\label{4.31}
  \{I_R ((N_R f_j )^\lambda)\}\text{ is bounded in }L^{\gamma}(\mathcal{P}_{4R}(0,0))
 \quad\text{for all } \gamma\in(1,\infty).
\end{equation}
Let $\hat q=p_2/\sigma$. Then by \eqref{4.22},
\[
\frac{1}{p}-\frac{1}{\hat q}=\frac{1}{p}-\frac{\sigma}{p_2}
=\frac{2-\varepsilon}{n+2}+\frac{1-\sigma}{p_2}.
\] 
Thus for $\sigma\le 1$ we have
\[
\frac{1}{p}-\frac{1}{\hat q}\ge\frac{2-\varepsilon}{n+2}>0
 \]
and for $\sigma>1$ it follows from \eqref{4.30} and \eqref{4.6} that
\begin{align*}
\frac{1}{p}-\frac{1}{\hat q}
&\ge\frac{2-\varepsilon}{n+2}-\frac{\sigma-1}{\frac{(n+2)\lambda}{n+2-\alpha}}\\
&\ge \frac{2-\varepsilon}{n+2}
-\frac
{\frac{2(n+2)-\alpha}{n}-\lambda-1}
{\frac{(n+2)\lambda}{n+2-\alpha}}\\
&=\frac{n+4-\alpha-\varepsilon}{(n+2)\lambda}
\left(\lambda-\left(\frac{n+4-\alpha}{n+4-\alpha-\varepsilon}\right)
\frac{n+2-\alpha}{n}\right)>0. 
\end{align*}
Thus defining $q\in(p,\hat q)$ by
\[
\frac{1}{q}=\frac{\frac{1}{p}+\frac{1}{\hat q}}{2}
\]
we have for $\sigma>0$ that
\[
\frac{1}{p}-\frac{1}{q}=\frac{1}{2}\left(\frac{1}{p}-\frac{1}{\hat
    q}\right)
\ge C_0(n,\lambda,\sigma,\alpha)>0.
\]
That is \eqref{4.21} holds.

Since $q\sigma/p_2<\hat q\sigma/p_2=1$ there exists
$\gamma\in(q,\infty)$ such that
\begin{equation}\label{4.32}
\frac{q}{\gamma}+\frac{q\sigma}{p_2}=1.
\end{equation}
Also
\[
\|(I_R ((N_R f_j )^\lambda ))^q \|_{\gamma/q}
=\| I_R ((N_R f_j )^\lambda )\|^{q}_{\gamma}
\]
and by \eqref{4.24}
\[
\|(N_R f_j )^{\sigma q}\|_{\frac{p_2}{\sigma q}}=\|(N_R f_j )^\sigma
\|^{q}_{p_2 /\sigma}\leq C\| f_j \|^{\sigma q}_{p}.
\]
 It therefore follows from \eqref{4.18}, \eqref{4.32}, 
H\"older's inequality, \eqref{4.31}, and \eqref{4.19} that \eqref{4.20} holds.
\end{proof}

We return now to the proof of Theorem \ref{thm4.1}.  By \eqref{4.14} the sequence
\begin{equation}\label{4.33}
 \{f_j \}\text{ is bounded in }L^1(\mathcal{P}_4 (0,0)).
\end{equation}
Starting with this fact and iterating Lemma \ref{lem4.1} a finite number of times ($m$ times is enough if $m>1/C_0$) we see that there exists $R_0 \in(0,1)$ such that the sequence
$$\{f_j \}\text{ is bounded in }L^p (\mathcal{P}_{4R_0}(0,0))$$
for some $p>(n+2)/2$.  Hence by Theorem \ref{thmB.2} the sequence
$\{N_{R_0}f_j \}$ is bounded in $L^\infty (\mathcal{P}_{4R_0}(0,0))$.
Thus \eqref{4.18} implies the sequence
\begin{equation}\label{4.34}
 \{f_j \}\text{ is bounded in }L^\infty (\mathcal{P}_{R_0}(0,0)).
\end{equation}
Since by Lemma \ref{lem2.8},
\begin{align*}
\iint_{\mathcal{P}_{4}(0,0)}&\Phi(-\eta,-\zeta)f_j(\eta,\zeta)\,d\eta\,d\zeta\\
&\le \iint_{\mathcal{P}_{R_0}(0,0)}\Phi(-\eta,-\zeta)f_j(\eta,\zeta)\,d\eta\,d\zeta
+\frac{C(n)}{R_0^{n/2}}\iint_{\mathcal{P}_{4}(0,0)\setminus\mathcal{P}_{R_0}(0,0)}
f_j(\eta,\zeta)\,d\eta\,d\zeta
\end{align*}
we see that \eqref{4.33} and \eqref{4.34} contradict \eqref{4.15}.
This contradiction completes the proof of Theorem \ref{thm4.1}.
\end{proof}

\section{The case $\lambda\geq\frac{n+2}{n}$}\label{sec5}

In this section we prove Theorems \ref{thm1.1}--\ref{thm1.4} when
$\lambda\geq\frac{n+2}{n}$.  For these values of $\lambda$, Remark
2.1 and the following theorem imply Theorem \ref{thm1.1}.

\begin{thm}\label{thm5.1}
 Suppose $u$ is a nonnegative solution of \eqref{2.10},\eqref{2.11} for some constants $\alpha\in(0,n+2)$,
 \begin{equation}\label{5.1}
  \lambda\geq\frac{n+2}{n}\quad\text{and}\quad 0\leq\sigma<1-\frac{\alpha-2}{n+2}\lambda.
 \end{equation}
 Then
 \begin{equation}\label{5.2}
  \max_{|x|\leq1}u(x,t)=o(t^{-\frac{n+2}{2\lambda}})\quad\text{as }t\to0^+ .
 \end{equation}
\end{thm}

\begin{proof}
 Let $v=u+1$.  Then by Lemma \ref{lem2.3} we have that \eqref{2.12}--\eqref{2.15} hold.  To prove \eqref{5.2} it clearly suffices to prove
 \begin{equation}\label{5.3}
  \max_{|x|\leq1}v(x,t)=o(t^{-\frac{n+2}{2\lambda}})\quad\text{as }t\to0^+ .
 \end{equation}
 
 Since increasing $\sigma$ increases the right side of the second inequality in \eqref{2.13}$_1$, we can assume instead of \eqref{5.1} that
 \begin{equation}\label{5.4}
  \lambda\geq\frac{n+2}{n}\quad\text{and}\quad 0<\sigma<1-\frac{\alpha-2}{n+2}\lambda
 \end{equation}
 which implies
 \begin{equation}\label{5.5}
  \frac{\sigma}{\lambda}<\frac{2-\alpha}{n+2}+\frac{1}{\lambda}\leq\frac{2-\alpha}{n+2}+\frac{n}{n+2}=\frac{n+2-\alpha}{n+2}.
 \end{equation}
 By \eqref{5.4} there exists $\varepsilon=\varepsilon(n,\lambda,\sigma,\alpha)\in(0,1)$ such that
 \begin{equation}\label{5.6}
  \alpha+\varepsilon<n+2\quad\text{and}\quad \sigma<1-\frac{\alpha+\varepsilon-2}{n+2}\lambda
 \end{equation}
 which implies
 \begin{equation}\label{5.7}
  \frac{\sigma-1}{\lambda}<\frac{2-\alpha-\varepsilon}{n+2}.
 \end{equation}
 
 Part of the proof of Theorem \ref{thm5.1} will consist of two lemmas, the first of which is the following.

\begin{lem}\label{lem5.1}
  Suppose $\Omega$ is a bounded open subset of $\mathbb{R}^n \times\mathbb{R}$ and
  \begin{equation}\label{5.8}
   p\in\left[\lambda,\frac{(n+2)\lambda}{n+2-\alpha-\varepsilon}\right).
  \end{equation}
  Then for all $w\in L^p (\Omega)$ we have
  \begin{equation}\label{5.9}
   \left\|\left(\iint_\Omega \Phi(\cdot-y,\cdot-s)^{\alpha/n}w(y,s)^\lambda \,dy\,ds\right)w^\sigma \right\|_{L^{p_3}(\Omega)}\leq C\| w\|^{\lambda+\sigma}_{L^p (\Omega)}
  \end{equation}
  where
  \begin{equation}\label{5.10}
   \frac{1}{p_3}=\frac{\lambda+\sigma}{p}-\frac{n+2-\alpha-\varepsilon}{n+2}
  \end{equation}
  and $C=C(n,\lambda,\sigma,\alpha,\Omega,p)$ is a positive constant.  Moreover,
  \begin{equation}\label{5.11}
   p_3 >1.
  \end{equation}
 \end{lem}

 \begin{proof}
  Define $p_2$ by 
  $$\frac{\lambda}{p}-\frac{1}{p_2}=\frac{n+2-\alpha-\varepsilon}{n+2}.$$
  Then by \eqref{5.8} and \eqref{5.6}$_1$, $1\leq p/\lambda <p_2 <\infty$ and thus by Theorem \ref{thmB.2} we have, letting
  $$I(f)=\iint_\Omega \Phi(\cdot-y,\cdot-s)^{\alpha/n}f(y,s)\,dy\,ds,$$
  that 
  \begin{equation}\label{5.12}
   \| I(w^\lambda )\|_{L^{p_2}(\Omega)}\leq C\| w^\lambda \|_{L^{p/\lambda}(\Omega)}=C\| w\|^{\lambda}_{L^p (\Omega)}.
  \end{equation}
  Since $\frac{1}{p_3}=\frac{1}{p_2}+\frac{\sigma}{p}$ we have by
  H\"older's inequality that
  \begin{align*}
   \| I(w^\lambda )w^\sigma &\|^{p_3}_{L^{p_3}(\Omega)}
=\|(I(w^\lambda )w^\sigma )^{p_3}\|_{L^1 (\Omega)}\\
   &\leq\| I(w^\lambda )^{p_3}\|_{L^{p_2 /p_3}(\Omega)}\| w^{\sigma p_3}\|_{L^{\frac{p}{\sigma p_3}}(\Omega)}\\
   &=\| I(w^\lambda )\|^{p_3}_{L^{p_2}(\Omega)}\| w\|^{\sigma p_3}_{L^p (\Omega)}.
  \end{align*}
  Thus \eqref{5.9} follows from \eqref{5.12}.
  
  Also from \eqref{5.8} and \eqref{5.7} we find that
  $$\frac{1}{p_3}\leq\frac{\lambda+\sigma}{\lambda}-\frac{n+2-\alpha-\varepsilon}{n+2}=\frac{\sigma}{\lambda}+\frac{\alpha+\varepsilon}{n+2}<\frac{1}{\lambda}+\frac{2}{n+2}.$$
  Thus \eqref{5.11} follows from \eqref{5.4}$_1$.
 \end{proof}
 
 We now continue with the proof of Theorem \ref{thm5.1}.  Suppose for
 contradiction that \eqref{5.3} is false.  Then there exists a
 sequence $\{(x_j ,t_j )\}\subset\overline{B_1 (0)}\times(0,1/2)$
 and $x_0 \in\overline{B_1 (0)}$ such that
 \begin{equation}\label{5.13}
  (x_j ,t_j )\to (x_0 ,0)\quad\text{as }j\to\infty
 \end{equation}
 and
 \begin{equation}\label{5.14}
  \liminf_{j\to\infty}t^{\frac{n+2}{2\lambda}}_{j}v(x_j ,t_j )>0.
 \end{equation}
 Define $p_3 >0$ by
 $\frac{1}{p_3}=\frac{\alpha+\varepsilon}{n+2}+\frac{\sigma}{\lambda}$.
 Then by \eqref{2.12}, \eqref{2.13} and Lemma \ref{lem5.1} with
 $\Omega=\mathcal{P}_8 (0,8),\,p=\lambda$, and $w=v$ we have $p_3 >1$ and
 $Hv\in L^{p_3}(\mathcal{P}_4 (0,4)$.  Hence defining $p_4$ by
 $\frac{1}{p_3}+\frac{1}{p_4}=1$, using H\"older's inequality, and
 making the change of variables
 $$
  \begin{array}{cc}
   x=x_j +\sqrt{Rt_j}\xi, & t=t_j +Rt_j \tau\\
   y=x_j +\sqrt{Rt_j}\eta, & s=t_j +Rt_j \zeta
  \end{array}
 $$
 we see for $R\in(0,1]$ that
 \begin{align*}
 &\sup_{(x,t)\in\mathcal{P}_{Rt_j /4}(x_j ,t_j )}
\iint_{\mathcal{P}_4 (0,4)\backslash\mathcal{P}_{Rt_j /2}(x_j ,t_j )}\Phi(x-y,t-s)Hv(y,s)\,ds\\
  &\leq\sup_{(x,t)\in\mathcal{P}_{Rt_j /4}(x_j ,t_j )}\left(\iint_{\mathcal{P}_4 \backslash\mathcal{P}_{Rt_j /2}(x_j ,t_j )}\Phi(x-y,t-s)^{p_4}\,dy\,ds\right)^{1/p_4}\| Hv\|_{L^{p_3}(\mathcal{P}_4 (0,4))}\\
  &\leq C\sup_{(\xi,\tau)\in\mathcal{P}_{1/4}(0,0)}\left(\iint_{\mathbb{R}^n \times\mathbb{R}\backslash\mathcal{P}_{1/2}(0,0)}\left(\frac{1}{(Rt_j )^{n/2}}\right)^{p_4}\Phi(\xi-\eta,\tau-\zeta)^{p_4}(Rt_j )^{\frac{n+2}{2}}\,d\eta\,d\zeta\right)^{1/p_4}\\
  &=C\left(\frac{1}{Rt_j}\right)^{(\frac{np_4}{2}-\frac{n+2}{2})\frac{1}{p_4}}\sup_{(\xi,\tau)\in\mathcal{P}_{1/4}(0,0)}\left(\iint_{\mathbb{R}^n \times\mathbb{R}\backslash\mathcal{P}_{1/2}(0,0)}\Phi(\xi-\eta,\tau-\zeta)^{p_4}\,d\eta \,d\zeta\right)^{1/p_4}\\
  &=C\left(\frac{1}{Rt_j}\right)^{\frac{n+2}{2\lambda}(\sigma-\frac{2-\alpha-\varepsilon}{n+2}\lambda)}
 \end{align*}
 where $C$ depends on neither $R$ nor $j$ and
 $$\sigma-\frac{2-\alpha-\varepsilon}{n+2}\lambda<1$$ 
 by \eqref{5.6}$_2$.
 
 Also, using \eqref{2.14}, Lemma \ref{lem2.8}, and the fact that $\mathcal{P}_{t_j /4}(x_j ,t_j )\subset\mathcal{P}_2 (0,2)$ we see for $R\in(0,1]$ that 
 \begin{align*}
  \sup_{(x,t)\in\mathcal{P}_{Rt_j /4}(x_j ,t_j )}&\iint_{\mathcal{P}_8 (0,8)\backslash\mathcal{P}_4 (0,4)}\Phi(x-y,t-s)Hv(y,s)\,dy\,ds\\
  &\leq C(n)\iint_{\mathcal{P}_8 (0,8)}Hv(y,s)\,dy\,ds<\infty.
 \end{align*}
 Thus by \eqref{2.12}, \eqref{2.15} and Lemma \ref{lem2.4} with $p=\lambda,\,\Omega\times(0,T)=B_2 (0)\times(0,4)$, and $K=\overline{B_{3/2}(0)}$ we have for $(x,t)\in\mathcal{P}_{Rt_j /4}(x_j ,t_j )$ and $R\in(0,1]$ that
 \begin{equation}\label{5.15}
  v(x,t)\leq\iint_{\mathcal{P}_{Rt_j /2}(x_j ,t_j )}\Phi(x-y,t-s)Hv(y,s)\,dy\,ds+\frac{\varepsilon_j}{(Rt_j )^{\frac{n+2}{2\lambda}}}
 \end{equation}
 for some sequence $\{\varepsilon_j \}\subset(0,1)$ which tends to zero as $j\to\infty$ and which depends in neither $(x,t)$ nor $R$.
 
 Also, for $(x,t)\in\mathcal{P}_{Rt_j /4}(x_j ,t_j )$ and $R\in(0,1]$ we have by \eqref{2.12} and Lemma \ref{lem2.8} that
 \begin{align*}
  \iint_{\mathcal{P}_8 (0,8)\backslash\mathcal{P}_{Rt_j /2}(x_j ,t_j )}\Phi(x-y,t-s)^{\alpha/n}v(y,s)^\lambda \,dy\,ds&\leq\left(C(n)\left(\frac{4}{Rt_j}\right)^{n/2}\right)^{\alpha/n}\| v^\lambda \|_{L^1 (\mathcal{P}_8 (0,8))}\\
  &=\frac{C}{(Rt_j )^{\alpha/2}}
 \end{align*}
 where $C$ depends on neither $(x,t),\,R$, nor $j$.  Thus for $(x,t)\in\mathcal{P}_{Rt_j /4}(x_j ,t_j )$ and $R\in(0,1]$ we have by \eqref{2.13} that
 \begin{equation}\label{5.16}
  0\leq Hv(x,t)\leq C\left(\frac{1}{(Rt_j )^{\alpha/2}}+\iint_{\mathcal{P}_{Rt_j /2}(x_j ,t_j )}\Phi(x-y,t-s)^{\alpha/n}v(y,s)^\lambda \,dy\,ds\right)v(x,t)^\sigma 
 \end{equation}
 where $C$ depends on neither $(x,t),\,R$, nor $j$.
 
 Next, making the change of variables
 \[
 v(y,s)=t^{-\frac{n+2}{2\lambda}}_{j}v_j (\eta,\zeta),
\]
\[
 x=x_j +\sqrt{t_j}\xi, \quad t=t_j +t_j \tau;\qquad
 y=x_j +\sqrt{t_j}\eta, \quad s=t_j +t_j \zeta,
\]
we obtain
 \begin{equation}\label{5.17}
  \iint_{\mathcal{P}_{1/2}(0,0)}v_j (\eta,\zeta)^\lambda \,d\eta\,d\zeta=\iint_{\mathcal{P}_{t_j /2}(x_j ,t_j )}v(y,s)^\lambda \,dy\,ds
 \end{equation}
 and from \eqref{5.15} and \eqref{5.16} we find for $(\xi,\tau)\in\mathcal{P}_{R/4}(0,0)$ and $R\in(0,1]$ that
 \begin{align}\label{5.18}
  \notag v_j(\xi,\tau)&\leq\iint_{\mathcal{P}_{{R/2}}(0,0)}\frac{1}{t^{n/2}_{j}}\Phi(\xi-\eta,\tau-\zeta)t^{-1}_{j}Hv_j (\eta,\zeta)t^{\frac{n+2}{2}}_{j}\,d\eta\,d\zeta +\frac{\varepsilon_j}{R^{\frac{n+2}{2\lambda}}}\\
  &=\iint_{\mathcal{P}_{R/2}(0,0)}\Phi(\xi-\eta,\tau-\zeta)Hv_j (\eta,\zeta)\,d\eta\,d\zeta+\frac{\varepsilon_j}{R^{\frac{n+2}{2\lambda}}},
 \end{align}
 where $\varepsilon_j \to0$ as $j\to\infty$ and $\varepsilon_j$ depends on neither $(\xi,\tau)$ nor $R$, and
 \begin{align}\label{5.19}
  \notag 0\leq Hv_j (\xi,\tau)&\leq Ct^{\frac{n+2}{2\lambda}+1-\frac{\alpha}{2}}_{j}\left(\frac{1}{R^{\alpha/2}}+\iint_{\mathcal{P}_{R/2}(0,0)}\Phi(\xi-\eta,\tau-\zeta)^{\alpha/n}v_j (\eta,\zeta)^\lambda \,d\eta\,d\zeta\right)\\
  \notag &\quad\times(t^{-\frac{(n+2)\sigma}{2\lambda}}_{j}v_j (\xi,\eta)^\sigma )\\
  &=C\hat{\varepsilon}_j\left(\frac{1}{R^{\alpha/2}}+\iint_{\mathcal{P}_{R/2}(0,0)}\Phi(\xi-\eta,\tau-\zeta)^{\alpha/n}v_j (\eta,\zeta)^\lambda \,d\eta\,d\zeta\right)v_j (\xi,\eta)^\sigma
 \end{align}
 where $C$ depends on neither $(\xi,\tau),\,R$, nor $j$ and 
 $$\hat{\varepsilon}_j :=t^{-\frac{n+2}{2\lambda}(\sigma-1+\frac{\alpha-2}{n+2}\lambda)}_{j}\to0\quad\text{as }j\to\infty$$
 by \eqref{5.1} and \eqref{5.13}.
 
 Also by \eqref{5.14} we have 
 \begin{equation}\label{5.20}
  \liminf_{j\to\infty}v_j (0,0)>0.
 \end{equation}
 
 To complete the proof of Theorem \ref{thm5.1} we will require the following lemma.
 \begin{lem}\label{lem5.2}
  Suppose the sequence
  \begin{equation}\label{5.21}
   \{v_j \}\text{ is bounded in }L^p (P_{R/2}(0,0))
  \end{equation}
  for some constants $R\in(0,1]$ and
  \begin{equation}\label{5.22}
   p\in\left[\lambda,\frac{(n+2)\lambda}{n+2-\alpha-\varepsilon}\right).
  \end{equation}
  Then either the sequence
  \begin{equation}\label{5.23}
   \{v_j \}\text{ tends to zero in }L^{\frac{(n+2)\lambda}{n+2-\alpha-\varepsilon}}(P_{R/8}(0,0))
  \end{equation}
  or there exists a positive constant $C_0 =C_0 (n,\lambda,\sigma,\alpha)$ such that the sequence
  \begin{equation}\label{5.24}
   \{v_j \}\text{ tends to zero in }L^q (P_{R/8}(0,0))
  \end{equation}
  for some $q\in(p,\infty)$ satisfying
  \begin{equation}\label{5.25}
   \frac{1}{p}-\frac{1}{q}\geq C_0 .
  \end{equation}
 \end{lem}
 
 \begin{proof}
   It follows from \eqref{5.19}, \eqref{5.21}, \eqref{5.22}, and Lemma
   \ref{lem5.1} that the sequence
  \begin{equation}\label{5.26}
   \{Hv_j \}\text{ tends to $0$ in }L^{p_3}(P_{R/4}(0,0))
  \end{equation}
  where $p_3$, defined by \eqref{5.10}, satisfies \eqref{5.11}.\\
\medskip  

\noindent\textbf{Case I.} Suppose $p_3 \geq\frac{n+2}{2}$.  Then by
  \eqref{5.26}, \eqref{5.18}, and Theorem \ref{thmB.2} we have the
  sequence
  \[
   \{v_j \}\text{ tends to zero in }L^q (P_{R/8}(0,0))\quad\text{for all }q>1
  \]
  which implies \eqref{5.23}.
 \medskip 
 
\noindent \textbf{Case II.}  Suppose $p_3 <\frac{n+2}{2}$.  Define $q$ by
  \begin{equation}\label{5.27}
   \frac{1}{p_3}-\frac{1}{q}=\frac{2}{n+2}.
  \end{equation}
  Then by \eqref{5.11}
  $$1<p_3 <q<\infty.$$
  Hence by \eqref{5.26}, \eqref{5.18} and Theorem \ref{B.1} we have
  \eqref{5.24} holds.
  
  Also by \eqref{5.27}, \eqref{5.10},  \eqref{5.22}, and \eqref{5.7} we get
  \begin{align*}
   \frac{1}{p}-\frac{1}{q}
&=\frac{1}{p}+\frac{2}{n+2}-\frac{1}{p_3}
=\frac{1}{p}+\frac{2}{n+2}-\frac{\sigma}{p}-\frac{\lambda}{p}+1-\frac{\alpha+\varepsilon}{n+2}\\
   &=-\frac{\lambda+\sigma-1}{p}+1-\frac{\alpha+\varepsilon-2}{n+2}\\
   &\geq\frac{1-(\lambda+\sigma)}{\lambda}+1+\frac{2-\alpha-\varepsilon}{n+2}>0.
  \end{align*}
  Thus \eqref{5.25} holds.
 \end{proof}

We now return to the proof of Theorem \ref{thm5.1}.  By \eqref{2.12} and
 \eqref{5.17}, the sequence $\{v_j \}$ tends to zero in
 $L^\lambda (P_{1/2}(0,0))$.  Starting with this fact on iterating
 Lemma \ref{lem5.2} a finite number of times we see that the sequence
 \begin{equation}\label{5.28}
  \{v_j \}\text{ tends to zero in }L^p (\mathcal{P}_{R/2}(0,0))
 \end{equation}
 for some $R\in(0,1)$ and for some
 \begin{equation}\label{5.29}
  p>\frac{(n+2)\lambda}{n+2-\alpha}.
 \end{equation}
 Hence the sequence $\{v^{\lambda}_{j}\}$ tends to zero in $L^{p/\lambda}(\mathcal{P}_{R/2}(0,0))$ and $\frac{p}{\lambda}>\frac{n+2}{n+2-\alpha}$.  Thus by Theorem \ref{thmB.2}, the sequence whose $j$th term is the integral on the right side of \eqref{5.19}, tends to zero in
 $L^\infty (\mathcal{P}_{R/2}(0,0))$.  So by \eqref{5.19}
 \begin{equation}\label{5.30}
  0\leq Hv_j <Cv^{\sigma}_{j}\quad\text{in }\mathcal{P}_{R/4}(0,0)
 \end{equation}
 where $C$ does not depend on $j$.  Hence by \eqref{5.28} the sequence $\{Hv_j \}$ tends to zero in $L^{p/\sigma}(\mathcal{P}_{R/4}(0,0))$ and by \eqref{5.29} and \eqref{5.5}
 $$\frac{p}{\sigma}>\frac{(n+2)\lambda}{(n+2-\alpha)\sigma}>\left(\frac{n+2}{n+2-\alpha}\right)^2 >1.$$
 Thus by \eqref{5.18} and Theorem \ref{thmB.2} the sequence 
 \begin{equation}\label{5.31}
  \{v_j \}\text{ tends to zero in }L^q (\mathcal{P}_{R/8}(0,0))\text{ where }q=
  \begin{cases}
   \infty, & \text{if }\frac{p}{\sigma}\geq\frac{n+2}{2-\varepsilon}\\
   \frac{1}{\frac{\sigma}{p}-\frac{2-\varepsilon}{n+2}}, & \text{if }\frac{p}{\sigma}<\frac{n+2}{2-\varepsilon}.
  \end{cases}
 \end{equation}
 However the possibility that $q=\infty$ is ruled out by \eqref{5.20}.  Hence we can assume $\frac{p}{\sigma}<\frac{n+2}{2-\varepsilon}$.  Then by \eqref{5.31},
  $$\frac{1}{p}-\frac{1}{q}=\frac{1-\sigma}{p}+\frac{2-\varepsilon}{n+2}.$$
  Thus, if $\sigma\in(0,1]$ then 
  $$\frac{1}{p}-\frac{1}{q}>\frac{1}{n+2}.$$
  On the other hand, if $\sigma>1$ then by \eqref{5.29} and \eqref{5.7}
  \begin{align*}
   \frac{1}{p}-\frac{1}{q}&=\frac{2-\varepsilon}{n+2}-\frac{\sigma-1}{p}\\
   &>\frac{2-\varepsilon}{n+2}-\frac{\sigma-1}{\lambda}\frac{n+2-\alpha}{n+2}\\
   &>\frac{2-\varepsilon}{n+2}-\frac{2-\alpha-\varepsilon}{n+2}
   =\frac{\alpha}{n+2}.
  \end{align*}
  Thus for $\sigma>0$ we have
  $$\frac{1}{p}-\frac{1}{q}>C(n,\alpha)>0.$$
  Hence, after a finite number of iterations of the procedure of going
  from \eqref{5.28} to \eqref{5.31} we see that the sequence $\{v_j\}$
  tends to zero in $L^\infty (\mathcal{P}_{\hat R}(0,0))$ for some
  $\hat R\in(0,R)$ which again contrdicts \eqref{5.20}. This completes
  the proof of Theorem \ref{thm5.1}.
\end{proof}

The following theorem implies Theorem \ref{thm1.2}.
\begin{thm}\label{thm5.2}
 Suppose
 \begin{equation}\label{5.32}
  \lambda\geq\frac{n+2}{n}\quad\text{and}\quad \gamma=\frac{n+2-\varepsilon}{2\lambda}
 \end{equation}
 for some $\varepsilon\in(0,1)$.  Then the function
 \begin{equation}\label{5.33}
  u(x,t)=\Psi(x,t):=\int_{\mathbb{R}^n}\Phi(x-y,t)|y|^{-2\gamma}dy
 \end{equation}
 is a $C^\infty$ positive solution of
 \begin{equation}\label{5.34}
  Hu=0\quad\text{in }\mathbb{R}^n \times(0,\infty)
 \end{equation}
 such that
 \begin{equation}\label{5.35}
  u\in L^\lambda (\mathbb{R}^n \times(0,T))\quad\text{for all }T>0,
 \end{equation}
 \begin{equation}\label{5.36}
  t^\gamma u(0,t)=u(0,1)\quad\text{for }0<t<\infty,
 \end{equation}
 and
 \begin{equation}\label{5.37}
  t^\gamma u(x,t)\text{ is bounded between positive constants}
 \end{equation}
 on $\{(x,t)\in\mathbb{R}^n \times(0,t):|x|<\sqrt{t}\}$.
\end{thm}

\begin{proof}
 By \eqref{5.32} we have $2\gamma<n$.  Thus \eqref{5.33} is a $C^\infty$ positive solution of \eqref{5.34}.
 
 For $a>0$ and $(x,t)\in\mathbb{R}^n \times(0,\infty)$ we find making the change of variables $y=az$ that
 \begin{align}\label{5.38}
  \notag u(ax,a^2 t)&=\int_{\mathbb{R}^n}\Phi(ax-y,a^2 t)|y|^{-2\gamma}dy\\
  \notag &=\int_{\mathbb{R}^n}\Phi(ax-az,a^2 t)a^{-2\gamma}|z|^{-2\gamma}a^n dz\\
  \notag &=a^{-2\gamma}\int_{\mathbb{R}^n}\Phi(x-z,t)|z|^{-2\gamma}dz\\
  &=a^{-2\gamma}u(x,t).
 \end{align}
 Taking $x=0$ and $t=1$ in \eqref{5.38} we get
 \begin{equation}\label{5.39}
  u(0,a^2 )=a^{-2\gamma}u(0,1)\quad\text{for all }a>0.
 \end{equation}
 Thus \eqref{5.36} holds.
 
 Taking $x\neq0,\,t>0$, and $a=1/|x|$ in \eqref{5.38} and using the fact that $u(x,t)$ is radially symmetric in $x$ about the origin we get
 \begin{equation}\label{5.40}
  u(x,t)=a^{2\gamma}u(ax,a^2 t)=|x|^{-2\gamma}u(e_1 ,\frac{t}{|x|^2})
=|x|^{-2\gamma}g\left(\frac{t}{|x|^2}\right)
 \end{equation}
 where $g(\zeta)=u(e_1 ,\zeta)$ and $e_1 =(1,0,...,0)\in\mathbb{R}^n$.  By \eqref{5.33},
 \begin{equation}\label{5.41}
  g(\zeta)\to1\quad\text{as }\zeta\to0^+
 \end{equation}
 and using \eqref{5.40} and \eqref{5.36} we obtain for $t>0$ that
 $$1=\lim_{x\to 0}\frac{u(x,t)}{u(0,t)}
=\lim_{x\to 0}\frac{|x|^{-2\gamma}g\left(\frac{t}{|x|^2}\right)}{u(0,1)t^{-\gamma}}
=\lim_{x\to0}\frac{1}{u(0,1)}\frac{g\left(\frac{t}{|x|^2}\right)}
{\left(\frac{t}{|x|^2}\right)^{-\gamma}}.$$  
 Thus
 \begin{equation}\label{5.42}
  \frac{g(\zeta)}{\zeta^{-\gamma}}\to u(0,1)\quad\text{as }\zeta\to\infty. 
 \end{equation}
 For $t>0$, it follows from \eqref{5.40}--\eqref{5.42} and \eqref{5.32} that
 \begin{align*}
  \int_{\mathbb{R}^n}u(x,t)^\lambda dx&=\int_{\mathbb{R}^n}|x|^{-2\lambda\gamma}g\left(\frac{t}{|x|^2}\right)^\lambda dx\\
  &\leq
    C\left[\int_{\sqrt{t}<|x|}|x|^{-2\lambda\gamma}dx
+\int_{|x|<\sqrt{t}}|x|^{-2\lambda\gamma}\left(\frac{t}{|x|^2}\right)^{-\gamma\lambda}dx\right]\\
  &\leq Ct^{-1+\varepsilon/2}
 \end{align*}
 which implies \eqref{5.35}.
 
 Making the change of variables
 $$x=\sqrt{t}\xi\quad\text{and}\quad y=\sqrt{t}\eta$$
 in \eqref{5.33} we get
 $$u(x,t)=\frac{1}{t^\gamma}U\left(\frac{x}{\sqrt{t}}\right)\quad\text{for }(x,t)\in\mathbb{R}^n \times(0,\infty)$$
 where
 $$U(\xi)=\frac{1}{(4\pi)^{n/2}}\int_{\mathbb{R}^n}e^{-|\xi-\eta|^2 /4}|\eta|^{-2\gamma}d\eta.$$
 Thus since $U(\xi)$ is bounded between positive constants for $|\xi|\leq1$, we find that \eqref{5.37} holds.
\end{proof}

The following theorem implies Theorems \ref{thm1.3} and \ref{thm1.4}
when $\lambda\geq(n+2)/n$.
\begin{thm}\label{thm5.3}
 Suppose $\alpha,\,\lambda$, and $\sigma$ are constants satisfying
 \begin{equation}\label{Q1}
  \alpha\in(0,n+2),\quad\lambda\geq\frac{n+2}{n},\quad\sigma\geq0,\quad\text{and}\quad \sigma>1+\frac{2-\alpha}{n+2}\lambda.
 \end{equation}
 Let $\varphi:(0,1)\to(0,\infty)$ be a continuous function satisfying
 $$\lim_{t\to0^+}\varphi(t)=\infty.$$
 Then there exists a positive function
 \begin{equation}\label{Q2}
  u\in C^\infty (\mathbb{R}^n \times(0,1))\cap L^\lambda (\mathbb{R}^n \times(0,1))
 \end{equation}
 satisfying
 \begin{equation}\label{Q3}
  0\leq Hu\leq(\Phi^{\alpha/n}*u^\lambda )u^\sigma\quad\text{in }\mathbb{R}^n \times(0,1),
 \end{equation}
 where $*$ is the convolution operation in $\mathbb{R}^n \times(0,1)$, such that
 \begin{equation}\label{Q4}
  u(0,t)\neq O(\varphi(t))\quad\text{as }t\to0^+ .
 \end{equation}
\end{thm}

\begin{proof}
 By scaling $u$ and noting by \eqref{Q1} that $\sigma+\lambda\neq1$ we see that it suffices to prove Theorem \ref{thm5.3} with \eqref{Q3} replaced with the weaker statement that there exists a positive constant $C=C(n,\lambda,\sigma,\alpha)$ such that $u$ satisfies
 \begin{equation}\label{Q5}
  0\leq Hu\leq C(\Phi^{\alpha/n}*u^\lambda )u^\sigma \quad\text{in }\mathbb{R}^n \times(0,1),
 \end{equation}
 where $(*)$ is the convolution operation in $\mathbb{R}^n \times(0,1)$.
 
 By \eqref{Q1} there exists $\varepsilon=\varepsilon(n,\lambda,\sigma,\alpha)\in(0,1)$ such that
 \begin{equation}\label{Q6}
  2\varepsilon<\alpha
 \end{equation}
 and
 \begin{equation}\label{Q7}
  \sigma>1+\frac{2-\alpha+2\varepsilon}{n+2-2\varepsilon}\lambda.
 \end{equation}
 Let
 \begin{equation}\label{Q8}
  \gamma=\frac{n+2-\varepsilon}{2\lambda}\quad\text{and}\quad p=\frac{2\lambda}{n+2-2\varepsilon}.
 \end{equation}
 Then
 \begin{equation}\label{Q9}
  \gamma p>1.
 \end{equation}
 Let $\{T_j\}\subset(0,1)$ be a sequence such that $T_j \to0$ as $j\to\infty$.  Define $w_j :(-\infty,T_j )\to(0,\infty)$ by
 \begin{equation}\label{Q10}
  w_j (t)=(T_j -t)^{-1/p}
 \end{equation}
 and define $t_j \in(0,T_j )$ by
 \begin{equation}\label{Q11}
  w_j (t_j )=t^{-\gamma}_{j}.
 \end{equation}
 Then
 \begin{equation}\label{Q12}
  \frac{T_j -t_j}{t_j}=\frac{w_j (t_j )^{-p}}{t_j}=t^{\gamma p-1}_{j}\to0\quad\text{as }j\to\infty
 \end{equation}
 by \eqref{Q9}.
 
 Choose $a_j \in((t_j+T_j)/2,T_j )$ such that $w_j (a_j )>j\varphi(a_j )$.  Then
 \begin{equation}\label{Q13}
  \frac{w_j (a_j )}{\varphi(a_j )}\to\infty\quad\text{as }j\to\infty.
 \end{equation}
 Let $h_j (s)=\sqrt{a_j-s}$ and $H_j (s)=\sqrt{a_j +\varepsilon_j -s}$ where $\varepsilon_j >0$ satisfies
 \begin{equation}\label{Q14}
  a_j +2\varepsilon_j <T_j, \quad t_j-\varepsilon_j >t_j/2, \quad
\varepsilon_j <T^{2}_{j}, \quad \text{and} \quad w_j (t_j -\varepsilon_j )>\frac{w_j (t_j )}{2}.
 \end{equation}
 Define
 \begin{align*}
  &\omega_j =\{(y,s)\in\mathbb{R}^n \times\mathbb{R}:|y|<h_j (s)\quad\text{and}\quad t_j <s<a_j \},\\
  &\Omega_j =\{(y,s)\in\mathbb{R}^n \times\mathbb{R}:|y|<H_j (s)\quad\text{and}\quad t_j -\varepsilon_j <s<a_j +\varepsilon_j \}.
 \end{align*}
 By taking a subsequence we can assume the sets $\Omega_j$ are pairwise disjoint.
 
 Let $\chi_j :\mathbb{R}^n \times\mathbb{R}\to[0,1]$ be a $C^\infty$ function such that $\chi_j \equiv1$ in $\omega_j$ and $\chi_j \equiv0$ in $\mathbb{R}^n \times\mathbb{R}\backslash\Omega_j$.  Define $f_j ,\,u_j :\mathbb{R}^n \times\mathbb{R}\to[0,\infty)$ by
 \begin{equation}\label{Q15}
  f_j (y,s)=\chi_j (y,s)w^{\prime}_{j}(s)
 \end{equation}
 and
 \begin{equation}\label{Q16}
  u_j (x,t)=\iint_{\mathbb{R}^n \times\mathbb{R}}\Phi(x-y,t-s)f_j (y,s)\,dy\,ds.
 \end{equation}
 Then $f_j$ and $u_j$ are $C^\infty$ and
 \begin{equation}\label{Q17}
  Hu_j =f_j \quad\text{in }\mathbb{R}^n \times\mathbb{R}.
 \end{equation}
By Theorem \ref{thmB.2} with $p=n+2$ and $q=\infty$ we see that
 \begin{align}\label{Q18}
  \notag\biggl\|\iint_{\Omega_j \backslash \omega_j}&\Phi(x-y,t-s)
w^{\prime}_{j}(s)\,dy\,ds\biggr\|_{L^\infty (\mathbb{R}^n \times(0,1))}\\
  \notag &=\biggl\|\iint_{\mathbb{R}^n \times(0,1)}
\Phi(x-y,t-s)\chi_{\Omega_j \backslash \omega_j}(y,s)
w^{\prime}_{j}(s)\,dy\,ds\biggr\|_{L^\infty (\mathbb{R}^n \times(0,1))}\\
  \notag &\leq C_n \| w^{\prime}_{j}(s)\|_{L^{n+2}(\Omega_j \backslash \omega_j )}\\
  &\leq w_j (t_j )
 \end{align}
 provided we decrease $\varepsilon_j$ if necessary because $|\Omega_j \backslash \omega_j |\to0$ as $\varepsilon_j \to0$.

 Also, it follows from \eqref{Q1}$_2$, \eqref{Q8}$_1$, \eqref{5.37},
 \eqref{Q14}$_1$, \eqref{Q12}, and \eqref{Q11} that there exists a
 positive constant $M$, independent of $j$, such that for
 $(x,t)\in\Omega_j$ we have
 \begin{equation}\label{Q19}
  M\Psi(x,t)>\frac{2^{\gamma+1}}{t^\gamma}>\frac{2^{\gamma+1}}{T^{\gamma}_{j}}>\frac{2^{\gamma+1}}{(2t_j )^\gamma}=2w_j (t_j ),
 \end{equation}
 provided we take a subsequence if necessary, where $\Psi$ is defined
 by \eqref{5.33}.
 
 In order to obtain a lower bound for $u_j$ in $\omega_j$, note first
 that for $s<t\leq a_j +\varepsilon_j$ and $|x|\leq H_j (t)$ we have by
 Lemma \ref{lem2.10} that
 \begin{equation}\label{Q20}
  \int_{|y|<H_j (s)}\Phi(x-y,t-s)\,dy\geq b
 \end{equation}
 for some constant
 \begin{equation}\label{Q21}
  b=b(n)\in(0,1).
 \end{equation}
 Next using \eqref{Q20} and \eqref{Q21}, we find for $(x,t)\in\Omega_j$ that
 \begin{align*}
  \iint_{\Omega_j}\Phi(x-y,t-s)w^{\prime}_{j}(s)\,dy\,ds&=\int^{t}_{t_j -\varepsilon_j}w^{\prime}_{j}(s)\left(\int_{|y|<H_j (s)}\Phi(x-y,t-s)\,dy\right)ds\\
  &\geq b(w_j (t)-w_j (t_j -\varepsilon_j ))\\
  &\geq bw_j (t)-w_j (t_j ).
 \end{align*}
 It therefore follows from \eqref{Q16}, \eqref{Q15}, and \eqref{Q18} that for $(x,t)\in\Omega_j$ we have 
 \begin{align}\label{Q22}
  \notag u_j (x,t)&\geq\iint_{\omega_j }\Phi(x-y,t-s)w^{\prime}_{j}(s)\,dy\,ds\\
  \notag &=\iint_{\Omega_j}\Phi(x-y,t-s)w^{\prime}_{j}(s)\,dy\,ds-\iint_{\Omega_j \backslash \omega_j}\Phi(x-y,t-s)w^{\prime}_{j}(s)\,dy\,ds\\
  &\geq bw_j (t)-2w_j (t_j ).
 \end{align}
 
 Define $\beta>0$ by
 \begin{equation}\label{Q23}
  \frac{1}{\beta}-\frac{1}{\lambda}=\frac{2}{n+2}.
 \end{equation}
 Then by \eqref{Q1}
 \begin{equation}\label{Q24}
  \frac{2}{n+2}<\frac{1}{\beta}=\frac{1}{\lambda}+\frac{2}{n+2}\leq\frac{n}{n+2}+\frac{2}{n+2}=1
 \end{equation}
 and by \eqref{Q8}
 $$p>\frac{2\lambda}{n+2}=\frac{2}{(n+2)/\lambda}=\frac{2}{\frac{n+2}{\beta}-2}=\frac{2\beta}{n+2-2\beta}.$$
 Thus
 \begin{equation}\label{Q25}
  \frac{n}{2}-\frac{\beta(p+1)}{p}+1=\frac{(n+2-2\beta)p-2\beta}{2p}>0.
 \end{equation}
 Next we slightly increase $\beta$ in such a way that \eqref{Q25} and the first inequality in \eqref{Q24} still hold.  Then instead of \eqref{Q23} and \eqref{Q24} we get
 \begin{equation}\label{Q26}
  \frac{1}{\beta}-\frac{1}{\lambda}<\frac{2}{n+2}
 \end{equation}
 and
 \begin{equation}\label{Q27}
  \frac{2}{n+2}<\frac{1}{\beta}<1
 \end{equation}
 respectively.
 
 From \eqref{Q15}, \eqref{Q10}, \eqref{Q14}$_1$ and \eqref{Q25} we find that
 \begin{align}
  \notag p^\beta &\iint_{\mathbb{R}^n \times\mathbb{R}}f_j (y,s)^\beta \,dy\,ds\leq p^\beta \iint_{\Omega_j}w^{\prime}_{j}(s)^\beta \,dy\,ds\\
  \notag &\leq p^\beta \int^{T_j}_{0}\int_{|y|<\sqrt{T_j -s}}w^{\prime}_{j}(s)^\beta \,dy\,ds\\
  \notag &=|B_1(0)|\int^{T_j}_{0}(T_j -s)^{n/2-\beta(p+1)/p}ds\\
  \label{Q28} &=|B_1(0)|\int^{T_j}_{0}\tau^{n/2-\beta(p+1)/p}\,d\tau\to0\quad\text{as }j\to\infty.
 \end{align}
 Hence by \eqref{Q16}, \eqref{Q26}, \eqref{Q27}, and Theorem \ref{thmB.2} we obtain
 \begin{equation}\label{Q29}
  \| u_j \|_{L^\lambda (\mathbb{R}^n \times(0,1))}\to0\quad\text{as }j\to\infty.
 \end{equation}
 Repeating the derivation of \eqref{Q28} with $\beta$ replaced with $1$, we find that
 $$\iint_{\mathbb{R}^n \times\mathbb{R}}f_j (y,s)\,dy\,ds\to0\quad\text{as }j\to\infty.$$
 Thus
 $$\iint_{\mathbb{R}^n \times\mathbb{R}}\sum^{\infty}_{j=1}f_j (y,s)d\eta ds<\infty$$
 provided we take a subsequence if necessary.  Hence, since the
 $C^\infty$ functions $f_j$ have disjoint supports, it follows from
 Theorem \ref{thm5.2} that the function
 $u:\mathbb{R}^n \times(0,\infty)\to(0,\infty)$ defined by
 \begin{equation}\label{Q30}
  u(x,t)=(M+1)\Psi(x,t)+\sum^{\infty}_{j=1}u_j (x,t)
 \end{equation}
 is $C^\infty$ and from \eqref{Q17} and Theorem \ref{thm5.2} we have 
 \begin{equation}\label{Q31}
  Hu=\sum^{\infty}_{j=1}f_j \quad\text{in }\mathbb{R}^n \times(0,\infty).
 \end{equation}
 By \eqref{Q29} and Theorem \ref{thm5.2},
 $$u\in L^\lambda (\mathbb{R}^n \times(0,1))$$
 provided we take a subsequence of $u_j$ if necessary.  Thus \eqref{Q2} holds.
 
 We now prove \eqref{Q5}.  By \eqref{Q31} and \eqref{Q15} we have 
 $$Hu\equiv0\quad\text{in }(\mathbb{R}^n \times(0,1))\backslash\bigcup^{\infty}_{j=1}\Omega_j.$$
 Hence to prove \eqref{Q5}, it suffice to prove there exists a positive constant $C=C(n,\lambda,\sigma,\alpha)$ such that
 \begin{equation}\label{Q32}
  0\leq Hu\leq C(\Phi^{\alpha/n}*u^\lambda )u^\sigma \quad\text{in }\Omega_j
 \end{equation}
 for $j=1,2,...$.
 
 By \eqref{Q30}, \eqref{Q22}, and \eqref{Q19} we have for $(x,t)\in\Omega_j$ that
 \begin{equation}\label{Q33}
  u(x,t)\geq(M+1)\Psi(x,t)+bw_j (t)-2w_j (t_j )\geq\Psi(x,t)+bw_j (t).
 \end{equation}
 Thus for $(x,t)\in\Omega_j$ we see by \eqref{Q31}, \eqref{Q15}, and \eqref{Q10} that
 \begin{align*}
  Hu(x,t)&=f_j (x,t)\leq w^{\prime}_{j}(t)=\frac{1}{p}w_j (t)^{1+p}\\
  &=\frac{1}{p}w_j(t)^{1+p-\sigma}w_j(t)^\sigma
\le \frac{1}{pb^\sigma}w_j(t)^{1+p-\sigma}u(x,t)^\sigma.
 \end{align*}
 Hence to prove \eqref{Q32} it suffices to show
 \begin{equation}\label{Q34}
  w_j (t)^{1+p-\sigma}< C\iint_{\mathbb{R}^n \times(0,1)}\Phi(x-y,t-s)^{\alpha/n}u(y,s)^\lambda \,dy\,ds\quad\text{for }(x,t)\in\Omega_j .
 \end{equation}
 
 Our proof of \eqref{Q34} consists of two cases.
\medskip

 \noindent \textbf{Case I.} Suppose
 \begin{equation}\label{Q35}
  (x,t)\in\Omega_j \quad\text{and}\quad t\leq\frac{T_j +t_j}{2}.
 \end{equation}
 Then using \eqref{Q14}$_4$, \eqref{Q8}$_2$, \eqref{Q1}$_2$ and the fact that $w_j$ is an increasing function we have
 \begin{align*}
  \frac{1}{2}&\leq\frac{w_j (t)}{2w_j (t_j -\varepsilon_j )}<\frac{w_j (t)}{w_j (t_j )}\\
  &\leq\left(\frac{T_j -\frac{T_j +t_j}{2}}{T_j -t_j}\right)^{-1/p}=2^{1/p}<2^{n/2}.
 \end{align*}
 Also by \eqref{Q11} and \eqref{Q12}
 $$\frac{w_j (t_j )}{T^{-\gamma}_{j}}=\left(\frac{T_j}{t_j}\right)^\gamma \in(1,2)$$
 provided we take a subsequence if necessary.  Thus \eqref{Q35} implies
 \begin{equation}\label{Q36}
  \frac{1}{2}<\frac{w_j (t)}{T^{-\gamma}_{j}}<2^{(n+2)/2}.
 \end{equation}
 Next making the change of variables
 \[
x=\sqrt{T_j}\xi,\quad t=T_j \tau;\qquad y=\sqrt{T_j}\eta,\quad s=T_j
\zeta;\qquad\text{and}\quad \hat{y}=\sqrt{T_j}\hat{\eta},
\]
 we get for $(y,s)\in\mathbb{R}^n \times(0,\infty)$ that
 \begin{align*}
  \Psi(y,s)&=\int_{\mathbb{R}^n}\Phi(y-\hat{y},s)|\hat{y}|^{-2\gamma}d\hat{y}\\
  &=\int_{\mathbb{R}^n}\frac{1}{T^{n/2}_{j}}\Phi(\eta-\hat{\eta},\zeta)T^{-\gamma}_{j}|\hat{\eta}|^{-2\gamma}T^{n/2}_{j}d\hat{\eta}\\
  &=T^{-\gamma}_{j}\Phi(\eta,\zeta)
 \end{align*}
 and thus for $(x,t)\in\Omega_j$ we obtain from \eqref{Q8}$_1$ that
 \begin{align}
  \notag \iint_{\mathbb{R}^n
   \times(0,1)}&\Phi(x-y,t-s)^{\alpha/n}\Psi(y,s)^\lambda \,dy\,ds\\
&=\iint_{\mathbb{R}^n \times(0,\tau)}\left(\frac{1}{T^{n/2}_{j}}\Phi(\xi-\eta,\tau-\zeta)\right)^{\alpha/n}(T^{-\gamma}_{j}\Psi(\eta,\zeta))^\lambda
   T^{\frac{n+2}{2}}_{j}\,d\eta\,d\zeta \notag\\
&\geq\frac{G(\xi,\tau)}{\sqrt{T_j}^{\alpha+2\gamma\lambda-(n+2)}}
=\frac{G(\xi,\tau)}{\sqrt{T_j}^{\alpha-\varepsilon}}\label{Q37}
 \end{align}
 where
 $$G(\xi,\tau):=\iint_{B_1 (0)\times(1/2,\tau)}\Phi(\xi-\eta,\tau-\zeta)^{\alpha/n}\Psi(\eta,\zeta)^\lambda \,d\eta\,d\zeta.$$
 Since by \eqref{Q35}$_1$, \eqref{Q14}$_1$, \eqref{Q12}, and \eqref{Q14}$_3$, 
 $$1>\tau=\frac{t}{T_j}\geq\frac{t_j -\varepsilon_j}{T_j}\to1\quad\text{as }j\to\infty$$
 we have by \eqref{Q35}$_1$ that
 $$|\xi|=\frac{|x|}{\sqrt{T_j}}<\frac{\sqrt{T_j -t}}{\sqrt{T_j}}=\sqrt{1-\frac{t}{T_j}}\to0\quad\text{as }j\to\infty.$$
 Thus, since $G$ is clearly continuous at $(\xi,\tau)=(0,1)$ and $G(0,1)>0$ we have by \eqref{Q37} that
 \begin{equation}\label{Q38}
  \iint_{\mathbb{R}^n \times(0,1)}\Phi(x-y,t-s)^{\alpha/n}\Psi(y,s)^\lambda \,dy\,ds\geq\frac{C}{\sqrt{T_j}^{\alpha-\varepsilon}}\quad\text{for }(x,t)\in\Omega_j
 \end{equation}
 provided we take a subsequence if necessary.
 
 Since by \eqref{Q7} and \eqref{Q8}$_2$
 $$\sigma-1>\left(\frac{2-\alpha+2\varepsilon}{n+2-2\varepsilon}\right)\lambda =p-\frac{\alpha-2\varepsilon}{n+2-2\varepsilon}\lambda>p-\frac{\alpha-\varepsilon}{n+2-\varepsilon}\lambda$$
 we have by \eqref{Q8}$_1$ that
 $$\gamma(1+p-\sigma)<\gamma\left(\frac{(\alpha-\varepsilon)\lambda}{n+2-\varepsilon}\right)=\frac{\alpha-\varepsilon}{2}.$$
 Thus \eqref{Q34} follows from \eqref{Q33}, \eqref{Q36}, and \eqref{Q38}.
 \medskip

 \noindent \textbf{Case II.} Suppose
 \begin{equation}\label{Q39}
  (x,t)\in\Omega_j \quad\text{and}\quad t\geq\frac{T_j +t_j}{2}.
 \end{equation}
 Then for $s<t$ we have by Lemma \ref{lem2.10} that
 $$\int_{|y|<H_j (s)}\Phi(x-y,t-s)^{\alpha/n}dy\geq\frac{C}{(t-s)^{(\alpha-n)/2}}$$
 for some positive constant $C=C(n,\alpha)$.  Thus for $(x,t)$ satisfying \eqref{Q39} we get
 \begin{align}
  \notag \iint_{\Omega_j}\Phi&(x-y,t-s)^{\alpha/n}w_j (s)^\lambda \,dy\,ds\geq\int^{t}_{t_j}w_j (s)^\lambda \left(\int_{|y|<H_j (s)}\Phi(x-y,t-s)^{\alpha/n}dy\right)ds\\
  \notag &\geq C\int^{t}_{t_j}\frac{ds}{(t-s)^{(\alpha-n)/2}(T_j -s)^{\lambda/p}}\\
  \notag &=\frac{C}{(T_j -t )^{(\alpha-n)/2+\lambda/p-1}}\int^{\frac{T_j -t_j}{T_j -t}}_{1}\frac{dz}{(z-1)^{(\alpha-n)/2}z^{\lambda/p}}\text{ where }T_j -s=(T_j -t)z\\
  \notag &\geq\frac{C}{(T_j -t)^{(\alpha-n)/2+\lambda/p-1}}\int^{2}_{1}\frac{dz}{(z-1)^{(\alpha-n)/2}z^{\lambda/p}}\\
  \label{Q40} &=\frac{C}{(T_j -t)^{(\alpha-n)/2+\lambda/p-1}}=\frac{C}{(T_j -t)^{(\alpha-2\varepsilon)/2}}
 \end{align}
 by \eqref{Q8}$_2$.
 
 Since by \eqref{Q7} and \eqref{Q8}$_2$
 $$\sigma-1>\frac{2-\alpha+2\varepsilon}{n+2-2\varepsilon}\lambda=p\frac{2-\alpha+2\varepsilon}{2}$$
 we see that
 $$\frac{1}{p}(1+p-\sigma)=1+\frac{1-\sigma}{p}<1+\frac{\alpha-2-2\varepsilon}{2}=\frac{\alpha-2\varepsilon}{2}.$$
 Thus \eqref{Q34} follows from \eqref{Q33}, \eqref{Q10}, and \eqref{Q40}.
 
 Finally from \eqref{Q33} and \eqref{Q13} we get
 $$\frac{u(0,a_j )}{\varphi(a_j )}\geq\frac{bw_j (a_j )}{\varphi(a_j )}\to\infty\quad\text{as }j\to\infty,$$
 which gives \eqref{Q4}.
\end{proof}

\appendix
%\appendixpage
%\addappheadtotoc

\section{Representation formula}\label{secA}
In this appendix we provide the following representation formula for
nonnegative supertemperatures.

\begin{thm}\label{thmA}
Suppose 
$0<R_1<R_2<R_3$ are constants and
$u$ is a $C^{2,1}$ nonnegative solution of
\begin{equation}\label{A.1}
 Hu\ge 0\quad \text{in } B_{\sqrt{R_3}}(0)\times (0,R_3) 
\subset \R^n \times \R,\ n\ge 1,
\end{equation}
where $Hu = u_t-\Delta u$ is the heat operator. Then
\begin{equation}\label{A.2}
 Hu \in L^1(B_{\sqrt{R_2}}(0) \times (0,R_2)),
\end{equation}
\begin{equation}\label{A.2.5}
u^\beta  \in L^1(B_{\sqrt{R_1}}(0) \times (0,R_1))\quad\text{for }
1\le\beta <\frac{n+2}{n}
\end{equation}
and there exist a finite positive Borel measure $\mu$ on
$B_{\sqrt{R_2}}(0)$ and a bounded function\\
$h\in C^{2,1}(B_{\sqrt{R_1}}(0) \times (-R_1,R_1))$ satisfying
\begin{alignat}{2}
 \label{A.3}
Hh &= 0  \quad \text{in } &B_{\sqrt{R_1}}(0)\times (-R_1,R_1)\\
\label{A.4}
 h &= 0  \quad \text{in } &B_{\sqrt{R_1}}(0)\times (-R_1,0] 
\end{alignat}
such that
\begin{equation}\label{A.5}
 u = N +v+h\quad \text{in } B_{\sqrt{R_1}}(0)\times (0,R_1)
\end{equation}
where
\begin{align}\label{A.6}
 N(x,t) &:= \int^{R_2}_0 \int_{|y|<\sqrt{R_2}} \Phi(x-y,t-s) Hu(y,s)\,dy\,ds,\\
\label{A.7}
v(x,t) &:= \int_{|y|<\sqrt{R_2}} \Phi(x-y,t)\,d\mu(y),
\end{align}
and $\Phi$ is the heat kernel \eqref{1.3}.
\end{thm}

\begin{proof}
  When $\beta=1$, $R_1=1$, $R_2=4$, and $R_3=16$, Theorem \ref{thmA}
  was proved in \cite{T2011}.  The proof of Theorem \ref{thmA} when
  $\beta=1$ is obtained by making straighforward changes to the proof
  in \cite{T2011}.  It remains only to prove \eqref{A.2.5} for
  $1<\beta<(n+2)/n$. To do this, it suffices by \eqref{A.5} to show
\begin{equation}\label{A.8}
N^\beta\in L^1(\R^n\times(0,R_2))\quad\text{for }1<\beta<(n+2)/n
\end{equation}
and 
\begin{equation}\label{A.9}
v^\beta\in L^1(\R^n\times(0,R_2))\quad\text{for }1<\beta<(n+2)/n.
\end{equation}
Theorem \ref{thmB.2} and \eqref{A.2} imply \eqref{A.8}.

Finally, for $t>0$, $\beta>1$, and $\beta'$ the conjugate H\"older
exponent of $\beta$ we have
\begin{align*}
\int_{\R^n}v(x,t)^\beta\,dx
&=\int_{\R^n}\left(\int_{|y|<\sqrt{R_2}}\Phi(x-y,t)\,d\mu(y)\right)^\beta dx\\
&\le \int_{\R^n}\left(\int_{|y|<\sqrt{R_2}}1^{\beta'}\,d\mu(y)\right)^{\beta/\beta'}
\int_{|y|<\sqrt{R_2}}\Phi(x-y,t)^\beta\,d\mu(y)\,dx\\
&=C\int_{|y|<\sqrt{R_2}}\left(\int_{\R^n}\Phi(x-y,t)^\beta\,dx\right)\,d\mu(y)\\
&=C\int_{|y|<\sqrt{R_2}}t^{-n\beta/2}\int_{\R^n}e^{-\frac{\beta|x-y|^2}{4t}}dx\,d\mu(y)\\
&=Ct^{n(1-\beta)/2}\quad\text{by Lemma \ref{lem2.2}}
\end{align*}
which implies \eqref{A.9}.
\end{proof}

\begin{rem}\label{remA}
If $u$ is a $C^{2,1}$ nonnegative solution of \eqref{A.1} where
$R_3>0$ then by Theorem \ref{thmA}, 
\[
u^\beta  \in L^1(B_{\sqrt{R}}(0) \times (0,R))\quad\text{for }
1\le\beta <\frac{n+2}{n}\text{ and }0<R<R_3.
\]
Thus the conclusion \eqref{A.2.5} in Theorem \ref{thmA} can be
replaced with 
\[
u^\beta  \in L^1(B_{\sqrt{R_2}}(0) \times (0,R_2))\quad\text{for }
1\le\beta <\frac{n+2}{n}.
\]
\end{rem}

\section{Heat potential estimates}\label{secB}

In this appendix we provide estimates for the heat potentials
\[
(J_\alpha f)(x,t)=\iint_{\R^n\times\R}\Phi(x-y,t-s)^{\frac{n+2-\alpha}{n}}f(y,s)\,dy\,ds
\]
and
\[
(V_\alpha f)(x,t)=\iint_{\Omega}
  \Phi(x-y,t-s)^{\frac{n+2-\alpha}{n}}f(y,s)\,dy\,ds, 
\]
where $\Phi$ is given by \eqref{1.3}, $\Omega=\R^n\times (a,b)$, and
$\alpha\in(0,n+2)$. The proofs of these estimates are given in
\cite[Appendix B]{GT2016}.

\begin{thm}\label{thmB.1}
 Suppose $0<\alpha<n+2$ and $1<p<\frac{n+2}{\alpha}$ are constants and
 $f:\R^n \times\R\to\R$ is a nonnegative measurable function.  Let 
$$q=\frac{(n+2)p}{n+2-\alpha p}.$$ 
 Then
 $$\| J_\alpha f\|_{L^q (\R^n \times\R)}\leq C\| f\|_{L^p (\R^n
   \times\R)}$$ 
where $C=C(n,p,\alpha)$ is a positive constant.
\end{thm}

\begin{thm}\label{thmB.2}
Let $p,q\in[1,\infty]$,  $\alpha$, and $\delta$ satisfy 
\begin{equation}\label{B.1}
0\le \delta =\frac{1}{p}-\frac{1}{q}<\frac{\alpha}{n+2}<1.
\end{equation}
Then $V_\alpha$ maps $L^p(\Omega)$ continuously into 
$L^q(\Omega)$ and for $f\in L^p(\Omega)$ we have
\[
\|V_\alpha  f\|_{L^q(\Omega)} \le M\|f\|_{L^p(\Omega)},
\]
where 
\[
M=C(b-a)^{(\alpha-(n+2)\delta)/2}\quad\text{for some constant }
C=C(n,\alpha,\delta)>0. 
\]
\end{thm}

Theorem \ref{thmB.2} is weaker than Theorem \ref{thmB.1} in
that the second inequality in \eqref{B.1} cannot be replaced with
equality. However it is stronger in that the cases $p=1$ and
$q=\infty$ are allowed.


\begin{thebibliography}{10}
\bibitem{CZ} H. Chen and F. Zhou, Classification of isolated
  singularities of positive solutions for Choquard equations,
  arXiv:1512.03181 [math.AP].

\bibitem{DA2010} J. T. Devreese and A. S. Alexandrov, Advances in
  polaron physics, Springer Series in Solid-State Sciences, vol. 159,
  Springer, 2010.

\bibitem{GT2016} M. Ghergu and S. D. Taliaferro, \emph{Isolated
  Singularities in Partial Differential Inequalities}, Cambridge
  University Press, 2016.

\bibitem{GT2016-2} M. Ghergu and S. D. Taliaferro, Pointwise bounds
  and blow-up for Choquard-Pekar inequalities at an isolated
  singularity, \emph{J. Differential Equations} \textbf{261} (2016),
  189--217.

\bibitem{J1995} K.R.W. Jones, Newtonian quantum gravity,
  \emph{Australian Journal of Physics} \textbf{48} (1995), 1055--1081.

\bibitem{L1976} E. H. Lieb, Existence and uniqueness of the minimizing
  solution of Choquard's nonlinear equation, Studies in Appl. Math. 57
  (1976/77), 93--105.

\bibitem{Lions1980} P.-L. Lions, The Choquard equation and related
  questions, \emph{Nonlinear Anal.} \textbf{4} (1980), 1063--1072.

\bibitem{Lions1984} P.-L. Lions, The concentration-compactness
  principle in the calculus of variations. The locally compact
  case. I, \emph{Ann. Inst. H. Poincar\'e Anal. Non Lin\'eaire}
  \textbf{1} (1984), 109--145.

\bibitem{MCL2011} C. Ma, W. Chen, and C. Li, Regularity of solutions
  for an integral system of Wolff type, \emph{Adv. Math.} \textbf{226}
  (2011), 2676--2699.

\bibitem{MZ2010} L. Ma and L. Zhao, Classification of positive
  solitary solutions of the nonlinear Choquard equation,
  \emph{Arch. Ration. Mech. Anal.} \textbf{195} (2010), 455--467.

\bibitem{MZ2012} M. Melgaard and F. Zongo, Multiple solutions of the
  quasirelativistic Choquard equation, \emph{J. Math. Phys.}
  \textbf{53} (2012), 033709.

\bibitem{MPT1998} I. M. Moroz, R. Penrose and P. Tod,
  Spherically-symmetric solutions of the Schr\"odinger-Newton
  equations, \emph{Classical Quantum Gravity} \textbf{15} (1998), 2733--2742.

\bibitem{MV2013a} V. Moroz and J. Van Schaftingen, Groundstates of
  nonlinear Choquard equations: Existence, qualitative properties and
  decay asymptotics, \emph{J. Funct. Anal.} \textbf{265} (2013), 153--184.

\bibitem{MV2013b} V. Moroz and J. Van Schaftingen, Nonexistence and
  optimal decay of supersolutions to Choquard equations in exterior
  domains, \emph{J. Differential Equations} \textbf{254} (2013), 3089--3145.

\bibitem{MV2015} V. Moroz and J. Van Schaftingen, Semi-classical
  states for the Choquard equation, \emph{Calc. Var. Partial
    Differential Equations} \textbf{52} (2015), 199--235.

\bibitem{P1954} S. Pekar, Untersuchung \"uber die Elektronentheorie
  der Kristalle, Akademie Verlag, Berlin, 1954.

\bibitem{QS} P. Quittner and P. Souplet, Superlinear parabolic
  problems, blow-up, global existence and steady states, Birkhauser,
  Basel, 2007.

\bibitem{S2002} S. G. Samko, Hypersingular Integrals and Their
  Applications, Taylor and Francis, London, 2002. 

\bibitem{T2011} S. D. Taliaferro, Initial blow-up of solutions of
  semilinear parabolic inequalities, \emph{J. Differential Equations}
  \textbf{250} (2011), 892--928.

\bibitem{WW2009} J. Wei and M. Winter, Strongly interacting bumps for
  the Schr\"odinger-Newton equations, \emph{J. Math. Phys.} \textbf{50} (2009),
  012905.

\bibitem{Zhuo2016} R. Zhuo, W. Chen, X. Cui, and Z. Yuan, Symmetry and
  non-existence of solutions for a nonlinear system involving the
  fractional Laplacian, \emph{Discrete Contin. Dyn. Syst.} \textbf{36}
  (2016), 1125--1141.


\end{thebibliography}
\end{document}